\documentclass[A4]{article}
\usepackage{amsmath, amssymb, amsthm, color, mathrsfs}

\theoremstyle{plain} 
\newtheorem{theorem}{\bf Theorem}[section]
\newtheorem{lemma}[theorem]{\bf Lemma}
\newtheorem{corollary}[theorem]{\bf Corollary}
\newtheorem{proposition}[theorem]{\bf Proposition}
\newtheorem{fact}[theorem]{\bf Fact}
\theoremstyle{definition} 
\newtheorem{definition}[theorem]{\bf Definition}
\newtheorem{remark}[theorem]{\bf Remark}

\def\glay{\color[cmyk]{0, 0, 0, 0.6}}
\makeatletter
\def\address#1#2{\begingroup
\noindent\parbox[t]{9.8cm}{%
\small{\scshape\ignorespaces#1}\par\vskip1ex
\noindent\small{\itshape E-mail}%
\/:\,#2\par\vskip4ex}\hfill%
\endgroup}%
\makeatother
\title{{\bf The gonality and the Clifford index of curves on a 
toric surface}}
\author{Ryo Kawaguchi}
\date{} 
\begin{document}
\maketitle
%
%
%
%
\begin{abstract}
We determine the gonality and the Clifford index for curves on a 
compact smooth toric surface. Moreover, it is shown that their gonality 
are computed by pencils on the ambient surface. From the geometrical view 
point, this means that the gonality can be read off from the lattice polygon 
associated to the curve. 
\end{abstract}
%
%
\section{Introduction}

In this paper, a {\it curve} always means a smooth irreducible projective 
curve over the complex number field unless otherwise mentioned. 
The {\it gonality} and the {\it Clifford index} are significant 
invariants in the study of linear systems on curves, which are 
defined by 
\begin{eqnarray*}
\begin{array}{rl}
{\rm gon}(C) &\!\!\! 
={\rm min}\{{\rm deg}f\mid f:C\rightarrow {\mathbb P}^{1}
\ {\rm surjective\ morphism}\}=
{\rm min}\{k\mid C\ {\rm has}\ g_{k}^{1}\},\\
{\rm Cliff}(C) &\!\!\! 
={\rm min}\{{\rm deg}D-2h^{0}(D)+2\mid D:{\rm divisor\ on}
\ C,\,h^{0}(D)\geq 2,\,h^{1}(D)\geq 2\}
\end{array}
\end{eqnarray*}
for a curve $C$. 
A curve of gonality $k$ is said to be $k$-gonal, and a pencil on a 
curve is called a {\it gonality pencil} if its degree is equal to the 
gonality. We cite several basic facts about the gonality and the 
Clifford index. 
Clearly, ${\rm gon}(C)=1$ means that $C$ is rational. 
The three statements ${\rm gon}(C)=2$, ${\rm Cliff}(C)=0$ and 
$C$ is elliptic or hyperelliptic are equivalent. Besides, 
${\rm Cliff}(C)=1$ holds if and only if $C$ is trigonal or a smooth 
plane quintic curve. On the other hand, Brill-Noether theory gives us 
upper bounds ${\rm gon}(C)\leq \left \lfloor\frac{g+3}{2}\right \rfloor$ 
and ${\rm Cliff}(C)\leq \left \lfloor\frac{g-1}{2}\right \rfloor$ 
for a curve of genus $g$, and equalities hold if the curve is general 
in moduli (cf. \cite{acgh}). 
Lastly, we mention a close relation between these 
two invariants: ${\rm gon}(C)-3\leq {\rm Cliff}(C)\leq 
{\rm gon}(C)-2$ (cf. \cite{copmar}). 

Although a considerable amount of work has revealed properties of 
the gonality and the Clifford index, it is still not easy to determine 
them for a given curve. Ideally, we would also like to see what kind of
gonality pencils does a curve have. In fact, however, it is 
difficult even to know whether the number of gonality pencils is finite 
or infinite. There are only two systematic results giving satisfactory 
answers for these questions: the cases of plane curves and curves on 
Hirzebruch surfaces (Theorem \ref{namba} and \ref{martens}). In this 
paper, we will study more general cases. Concretely, we consider curves 
on a compact smooth toric surface, and compute the gonality and the 
Clifford index (Theorem \ref{mthm}). 
From the point of view of the geometry of 
convex bodies, our result states that the gonality of such a curve 
coincides with the lattice width (see Definition \ref{baramosu}) of the 
lattice polygon associated to the curve. Namely, we can read off the 
gonality by observing the shape of the lattice polygon. 
This fact has been conjectured by Castryck and Cools in \cite{wou}, 
and our result gives an affirmative answer for it. 
In addition, Theorem \ref{mthm} tells us that 
apart from a few exceptional cases, a curve on a toric surface has a 
finite number of gonality pencils, and moreover, they become 
restrictions of preassigned $\mathbb P^{1}$-fibrations of the surface 
called toric fibrations. On the other hand, in the process to prove the 
main result, we also obtain the lower bound for the self-intersection 
number of a curve on a toric surface (Corollary \ref{zoma}). This 
formula by itself is suggestive and of wide application, although which 
is just a tool in this paper. 

Before we state the main theorem, let us review the cases of curves on 
the projective plane and Hirzebruch surfaces. First, the gonality and the 
Clifford index of plane curves are computed by the following formula. 
\begin{theorem}[\cite{nam,elms}]\label{namba}
Let $C$ be a smooth plane curve of degree $d$. If $d\geq 2$, then 
${\rm gon}(C)=d-1$ and any gonality pencil is cut out by lines 
passing through a fixed point on $C$. Furthermore, if $d\geq 5$, then 
${\rm Cliff}(C)=d-4$. 
\end{theorem}
%
%
%
%
Next, 
let $\Sigma_{e}$ be a Hirzebruch surface of degree $e$, and 
$\pi :\Sigma_{e}\rightarrow {\mathbb P}^{1}$ the ruling of $\Sigma_{e}$. 
Note that if $e=0$, then $\Sigma_{0}$ has another ruling $\pi '$ to 
$\mathbb P^{1}$ whose fiber is a section of $\pi$. In this case, 
we may assume that ${\rm deg}\pi |_{C}\leq {\rm deg}\pi '|_{C}$. 
For curves on Hirzebruch surfaces, Martens has computed the gonality. 
\begin{theorem}[\cite{mar2}]\label{martens}
Let $C$ be a smooth curve on $\Sigma_{e}$ which is not isomorphic 
to a smooth plane curve. Then ${\rm gon}(C)={\rm deg}\pi |_{C}$. 
In the case where $e\geq 1$, or $e=0$ and ${\rm deg}\pi |_{C}<
{\rm deg}\pi '|_{C}$, $\pi |_{C}$ is a unique gonality pencil on $C$. 
In the case where $e=0$ and ${\rm deg}\pi |_{C}={\rm deg}\pi '|_{C}$, 
$C$ has exactly two gonality pencils $\pi |_{C}$ and $\pi '|_{C}$. 
\end{theorem}
In the case of Theorem \ref{martens}, since the set of gonality pencils 
is finite, we obtain ${\rm Cliff}(C)={\rm gon}(C)-2$ (cf. \cite{copmar}). 
Here we recall that 
the projective plane and Hirzebruch surfaces are simplest examples of 
toric surfaces. Hence, as a natural continuation of the above results, 
we expect to determine the gonality and gonality pencils 
of a curve on a toric surface. In order to give a precise statement, 
we recall some terminology. Let $S$ be a compact smooth toric 
surface. Then $S$ contains an algebraic torus $T$ as a nonempty Zariski 
open set together with an action of $T$ on $S$, which is a natural 
extension of the torus action of $T$ on itself. 
A prime divisor on $S$ is called a {\it $T$-invariant divisor} if 
it is invariant with respect to the above action. 
Any $T$-invariant divisor is isomorphic to the projective line.  
A blowing-down of a $T$-invariant divisor gives a morphism from 
$S$ to another toric surface. We call a composition of such morphisms an 
{\it equivariant morphism}. It is known that if $S$ is not a projective 
plane, there exist a finite number of equivariant morphisms 
from $S$ to Hirzebruch surfaces. Hence, by composing such equivariant 
morphisms and the rulings of Hirzebruch surfaces, we  obtain a 
finite number of $\mathbb P^{1}$-fibrations of $S$. We call 
them {\it toric fibrations}. 
%
%
Now, we state the main theorem of this paper. 
\begin{theorem}\label{mthm}
Let $S$ be a compact smooth toric surface, and $C$ a $k$-gonal nef curve 
of genus $g\geq 2$ on $S$ which is not isomorphic to a smooth plane 
curve. Put $q={\rm min}\{{\rm deg}\hspace{0.3mm}\varphi |_{C} \mid 
\varphi :{\it toric\ fibration\ of}\ S\}$. Then $q$ is equal to the lattice 
width $($see Definition \ref{baramosu}$)$ of the lattice polygon associated 
to $C$, and the followings hold. 
\vspace{-1.5mm}
\begin{itemize}
\setlength{\itemsep}{-2pt}
\item[{\rm (i)}] 
If $(g,q)\neq (4,4),(5,4),(10,6)$, then any gonality pencil 
on $C$ is the restriction of a toric fibration of $S$. 
\item[{\rm (ii)}] 
The equalities $k=\left\{
\begin{array}{ll}
q & ((g,q)\neq (4,4))\\
q-1 & ((g,q)=(4,4))
\end{array}\right.$ hold. 
\vspace{1mm}
\item[{\rm (iii)}] 
If $(g,q)\neq (10,6)$, then ${\rm Cliff}(C)=k-2$. 
\item[{\rm (iv)}] 
If $(g,q)=(10,6)$, then $C$ is a complete intersection of two 
hypercubics in $\mathbb P^{3}$. 
\end{itemize}
\end{theorem}
%
%
In the case $(g,q)=(4,4)$, since $C$ is trigonal by (ii), we see 
that $C$ has one or two gonality pencils. In Section \ref{exam}, we will 
show that both cases can occur. 
In the case $(g,q)=(5,4)$, 
the gonality of $C$ achieves the maximum of the upper bound ${\rm 
gon}(C)\leq \left \lfloor\frac{g+3}{2}\right \rfloor$. It follows 
that $C$ has infinitely many gonality pencils. 
If $(g,q)=(10,6)$, by virtue of (iv) and Martens' work \cite{mar1}, 
we see that ${\rm gon}(C)=6$, ${\rm Cliff}(C)=3$ and $C$ has infinitely 
many gonality pencils. 
By a simple consideration, we can rewrite Theorem \ref{mthm} as follows:
%
%
%
%
\begin{corollary}\label{mthm2}
Let $S$ and $C$ be as in Theorem \ref{mthm}. 
\vspace{-1.5mm}
\begin{itemize}
\setlength{\itemsep}{-2pt}
\item[{\rm (i)}] 
If $(g,k)\neq (4,3),(5,4),(10,6)$, then any gonality pencil 
on $C$ is the restriction of a toric fibration of $S$. 
\item[{\rm (ii)}] 
If $(g,k)\neq (4,3)$, then $k=q$. 
\item[{\rm (iii)}] 
If $(g,k)\neq (10,6)$, then ${\rm Cliff}(C)=k-2$. 
\item[{\rm (iv)}] 
If $(g,k)=(10,6)$, then $C$ is a complete intersection of two 
hypercubics in $\mathbb P^{3}$. 
\end{itemize}
\end{corollary}
%
%
%
%
In Section \ref{ganirasu}, we review the theory of toric surfaces, which 
is the main stage of our study. The aim of Section \ref{lbd} is to reveal 
several properties of the self-intersection number of a curve on a toric 
surface, which will be utilized to prove the key 
proposition (Proposition \ref{kandata}) and Theorem \ref{mthm} in 
Section \ref{doruido}. Most proofs in Section \ref{lbd}, however, are 
just elementary and tedious computations for convex 
polygons in the affine plane. Hence the reader can skip them without 
losing the continuity of the paper. In Section \ref{exam}, as already 
mentioned after Theorem \ref{mthm}, we investigate trigonal curves of 
genus four. Finally, as an application of our results, we compute 
Weierstrass gap sequences at ramification points of a trigonal covering 
of $\mathbb P^{1}$ in Section \ref{appl}. In fact, gap sequences at 
such points are well studied and the classification of them has been 
already completed (\cite{cop1},\cite{cop2},\cite{kat},\cite{kim}). 
However, by combining Corollary \ref{mthm2} with results in \cite{kaw}, 
we can compute gap sequences in a completely different way and propose a 
novel geometric interpretation of the reason why the difference between 
types of gap sequences occurs. This approach can be adapted 
not only to trigonal curves but also to curves of higher gonality. Hence, 
as a generalization of the results in Section \ref{appl}, it is 
expected that we can classify Weierstrass gap sequences at ramification 
points of gonality pencils on a curve on a toric surface in the future. 
%
%
%
%
\section{Fans and lattice polygons}\label{ganirasu}

In this section, we briefly review basic notions in the theory of toric 
surfaces. 
For further background and 
applications of them, we refer the reader to \cite{oda} without explicit 
mention. We henceforth assume that a surface is always compact and smooth. 

For a toric surface $S$, there exists a {\it fan} 
$\Delta_{S}$, which is the division of $\mathbb R^{2}$ consisting of 
a finite number of half-lines starting from the origin 
called {\it cones} (see Fig. \ref{fan}). 
%
\begin{figure}[h]\hspace{-10mm}
\begin{picture}(400,127)
\multiput(180,25)(0,15){6}{\line(1,0){90}}
\multiput(195,13)(15,0){5}{\line(0,1){100}}
\multiput(180,69)(0,0.46){5}{\line(1,0){90}}
\multiput(224,13)(0.46,0){5}{\line(0,1){100}}
\multiput(186.9,30.6)(-0.26,0.26){6}{\line(1,1){75}}
\multiput(198.5,15)(-0.4,0.2){5}{\line(1,2){48.3}}
\multiput(225,69)(-0.2,0.4){5}{\line(-2,-1){43}}
\multiput(187.7,108.5)(-0.26,-0.26){6}{\line(1,-1){81}}
\multiput(225.4,71)(-0.4,-0.2){5}{\line(1,-2){29}}
\multiput(225.5,71)(-0.37,-0.25){5}{\line(2,-3){36}}
\multiput(224.9,69)(0.2,0.4){5}{\line(-2,1){43}}
\put(210,70){\circle*{5}}
\put(195,85){\circle*{5}}
\put(210,85){\circle*{5}}
\put(225,85){\circle*{5}}
\put(240,100){\circle*{5}}
\put(240,85){\circle*{5}}
\put(240,70){\circle*{5}}
\put(240,55){\circle*{5}}
\put(255,25){\circle*{5}}
\put(240,40){\circle*{5}}
\put(225,55){\circle*{5}}
\put(210,40){\circle*{5}}
\put(210,55){\circle*{5}}
\put(195,55){\circle*{5}}
\put(206,118.5){$\sigma(D_{1})$}
\put(241,118.5){$\sigma(D_{2})$}
\put(160,113){$\sigma(D_{d})$}
\put(271,108){\circle*{2}}
\put(275,101){\circle*{2}}
\put(278,93){\circle*{2}}
\put(170.8,101){\circle*{2}}
\put(168,93){\circle*{2}}
\put(167,85){\circle*{2}}
\put(219,0){$\Delta_{S}$}
\end{picture}
\caption{}\label{fan}
\end{figure}
%
Each cone $\sigma (D_{i})$ corresponds to a $T$-invariant 
divisor $D_{i}$, and a lattice point on $\sigma (D_{i})$ is called 
a {\it primitive element} if it is closest to the origin. We denote 
by $(x_{i},y_{i})$ the primitive element of $\sigma (D_{i})$, and by 
${\rm Pr}(S)$ the set of primitive elements of cones in $\Delta_{S}$. 
We assume that $(x_{1},y_{1})=(0,1)$. 
The assumption that $S$ is smooth means 
that $x_{i+1}y_{i}-y_{i+1}x_{i}=1$ for $i=1,\ldots ,d$, where we formally 
set $D_{d+1}=D_{1}$. The Picard group of $S$ is generated by the classes 
of $D_{1},\ldots ,D_{d}$. For instance, the canonical divisor of $S$ 
has the relation $K_{S}\sim -\sum_{i=1}^{d}D_{i}$. We next 
define a lattice polygon associated to a divisor on $S$, which 
is the essential notion in the study of curves on a toric surface. 
\begin{definition}
For a divisor $D=\sum_{i=1}^{d}n_{i}D_{i}$ on $S$, 
a {\it lattice polygon} associated to $D$ is defined by 
$\square_{D}=\{(z,w)\in {\mathbb R}^{2}\mid 
x_{i}z+y_{i}w\leq n_{i}\ {\rm for}\ 1\leq i\leq d\,\}$. 
\end{definition}
%
%
%
%
Lastly, we mention the structures of fibers of toric fibrations. 
We define ${\rm Pr}^{\ast}(S)=\{(x,y)\in {\rm Pr}(S)\mid 
(-x,-y)\in {\rm Pr}(S)\}$ and $M(u,v)=\{(x,y)\in {\rm Pr}(S)\mid 
uy-vx<0\}$ for integers $u$ and $v$. 
\begin{fact}\label{jinmencho}
For any primitive elements $(x_{i},y_{i})$ and $(x_{j},y_{j})$, we can 
uniquely write $(x_{j},y_{j})=
\alpha_{j}(x_{i},y_{i})+\beta_{j}(x_{i+1},y_{i+1})$ with some 
integers $\alpha_{j}$ and $\beta_{j}$. We can describe fibers of 
toric fibrations as follows: 
$$\textstyle \{{\it fibers\ of\ toric\ fibrations\ of}\ S\}=\bigg\{
\sum\limits_{(x_{j},y_{j})\in M(x_{i},y_{i})}|\beta_{j}|D_{j}\ \Big |
\ (x_{i},y_{i})\in {\rm Pr}^{\ast}(S)\bigg\}.$$
\end{fact}
%
%
%
\section{The lower bound for the self-intersection number}\label{lbd}

Let $S$ be a toric surface. In this section, we will find the evaluation 
formula for the self-intersection number of a curve on $S$ for 
later use in the proof the key proposition (Proposition \ref{kandata}). 
First, we extend the notion of coprime. 
%
%
\begin{definition}\label{youganmajin}
For non-negative integers $x$ and $y$, we write $(x,y)=1$ if they 
satisfy the following property: If either $x$ or $y$ is zero, then 
the other one is one. If both $x$ and $y$ are positive, then they are 
coprime. 
\end{definition}
\begin{definition}\label{baramosu}
Let $D$  be a divisor on $S$, and 
$x$ and $y$ integers with $(|x|,|y|)=1$. We denote by $n(x,y)$ the 
minimal integer satisfying $\{(z,w)\mid xz+yw\leq n(x,y)\}\supset 
\square_{D}$. For $x$, $y$ and $n(x,y)$, we define 
\begin{eqnarray*}
\begin{array}{c}
l(D,(x,y))=\{(z,w)\in {\mathbb R}^{2}\mid xz+yw=n(x,y)\}, \\
L(D,(x,y))=\{(z,w)\in {\mathbb R}^{2}\mid xz+yw\leq n(x,y)\}, \\
d(D,(x,y))=n(x,y)+n(-x,-y). 
\end{array}
\end{eqnarray*}
In particular, we call ${\rm min}\{d(D,(x,y))\mid 
\mbox{$x, y$\,:\,integers with $(|x|,|y|)=1$}\}$ the {\it lattice width} 
of $\square_{D}$. 
\end{definition}
\begin{remark}\label{kyatapira}
By definition, if $(z_{1},w_{1})\in L(D,(x,y))$ and $(z_{2},w_{2})\in 
L(D,(-x,-y))$, then $d(D,(x,y))\geq x(z_{1}-z_{2})+y(w_{1}-w_{2})$. 
In addition, an easy computation gives 
$d(D,(x,y))=\sum_{(x_{j},y_{j})\in M(x,y)}(x_{j}y-y_{j}x)D.D_{j}$. 
\end{remark}
%
%
Let $C$ be a curve on $S$. 
By Fact \ref{jinmencho} and Remark \ref{kyatapira}, we see that 
$q={\rm min}\{d(C,(x_{i},y_{i}))\mid 
(x_{i},y_{i})\in {\rm Pr}^{\ast}(S)\}$. Without 
loss of generality, we can assume that $(0,1)\in {\rm Pr}^{\ast}(S)$ 
and $d(C,(0,1))=q$. In the case where $\sharp {\rm Pr}^{\ast}(S)=2$, 
that is, $\Delta_{S}$ contains only one line passing through the 
origin, we can assume that $(1,0)\in {\rm Pr}(S)$ and $(x,y)\notin 
{\rm Pr}(S)$ if $y<{\rm max}\{0,-x\}$. On the other hand, in the case 
where $\sharp {\rm Pr}^{\ast}(S)\geq 4$, we can assume 
that $(1,0)\in {\rm Pr}^{\ast}(S)$ and $d(C,(1,y))\geq d(C,(1,0))$ for 
any $(1,y)\in {\rm Pr}^{\ast}(S)$. 
In this paper, we will keep the above assumptions for $C$ and 
$\Delta_{S}$, and put $q'=d(C,(1,0))$. 
\begin{lemma}
Let $C$ be a nef curve on $S$. For integers $x\,(\geq 1)$ and 
$y$ with $(|x|,|y|)=1$, the inequality $d(C,(x,y))\geq q'$ holds. 
\end{lemma}
\begin{proof}
(i) Consider the case where $\sharp {\rm Pr}^{\ast}(S)=2$. We denote 
by $O=(0,0)$ (resp. $P$) the vertex of $\square_{C}$ on $l(C,(0,-1))$ 
whose $z$-coordinate is minimal (resp. maximal). 
By a simple consideration, we see that $P=(q',0)$. Then 
by Remark \ref{kyatapira}, we have $d(C,(x,y))\geq xq'$. \\
(ii) In the case where $\sharp {\rm Pr}^{\ast}(S)\geq 4$, we will 
prove only the case where both $x$ and $y$ are positive. One can show 
other cases by a similar method. We denote by 
$n$ the maximal integer such that $(1,n)\in {\rm Pr}^{\ast}(S)$ and 
$nx-y\leq 0$, and define 
\begin{eqnarray*}
P_{1}=(z_{1},w_{1})=\left\{
\begin{array}{ll}
l(C,(1,n))\cap l(C,(1,0)) & (n\geq 1),\\
l(C,(1,0))\cap l(C,(0,-1)) & (n=0),
\end{array}\right.\hspace{3mm}\\
P_{2}=(z_{2},w_{2})=\left\{
\begin{array}{ll}
l(C,(-1,-n))\cap l(C,(-1,0)) & (n\geq 1),\\
l(C,(-1,0))\cap l(C,(0,1)) & (n=0).
\end{array}\right.
\end{eqnarray*}
In the case where $n\geq 1$, since 
$$d(C,(1,n))=z_{1}+nw_{1}-z_{2}-nw_{2}=q'+n(w_{1}-w_{2})\geq q',$$
we have $w_{1}\geq w_{2}$. Accordingly, we have 
$d(C,(x,y))\geq xz_{1}+yw_{1}-xz_{2}-yw_{2}\geq xq'$. Assume that 
$n=0$. Note that $x>y$ in this case. Since it is clear 
that $P_{1}\in L(C,(x,y))$ and $P_{2}\in L(C,(-x,-y))$,
by Remark \ref{kyatapira}, we have 
$d(C,(x,y))\geq xq'-yq\geq (x-y)q'$. 
\end{proof}
\begin{lemma}\label{ikkakuusagi}
For a nef curve $C$ on $S$, there exist a compact smooth toric 
surface $S_{0}$ and a curve $C_{0}$ on $S_{0}$ satisfying the 
following properties:\vspace{-4pt}
\begin{itemize}
\setlength{\itemsep}{-2pt}
\item[{\rm (i)}] 
$d(C_{0},(0,1))=q$, $d(C_{0},(1,0))=q'$ and 
$d(C_{0},(1,\pm 1))\geq q'$. 
\item[{\rm (ii)}] $C_{0}^{2}\leq C^{2}$. 
\item[{\rm (iii)}] The lattice polygon $\square_{C_{0}}$ has three or 
four vertices, and moreover, each of them is on one of the four 
lines $l(C_{0},(0,\pm 1))$, $l(C_{0},(\pm 1,0))$. 
\item[{\rm (iv)}] If $l(C_{0},(0,1))$ contains two distinct vertices, 
then they are $l(C_{0},(0,1))\cap l(C_{0},(1,0))$ and 
$l(C_{0},(0,1))\cap l(C_{0},(-1,0))$. A similar 
property holds for $l(C_{0},(0,-1))$. 
\item[{\rm (v)}] If $l(C_{0},(1,0))$ contains two distinct vertices, 
then one of them is $l(C_{0},(1,0))$ $\cap l(C_{0},$ $(0,1))$ or 
$l(C_{0},(1,0))\cap l(C_{0},(0,-1))$. A similar 
property holds for $l(C_{0},(-1,0))$. 
\end{itemize}
\end{lemma}
\begin{proof}
In this proof, we will gradually deform the polygon $\square_{C}$ 
toward $\square_{C_{0}}$. In the process of the deformation, we 
construct five polygons $\square_{C_{i}}$ ($i=1,\ldots ,5$). For 
simplicity, we abuse notation ``the properties (i) and 
(ii)'' for these curves $C_{i}$. 

We assume that the vertex of $\square_{C}$ on $l(C,(0,1))$ whose 
$z$-coordinate is minimal is the origin $O$, and 
denote by $P_{1}=(z_{1},w_{1})$ the vertex of $\square_{C}$ on 
$l(C,(1,0))$ whose $w$-coordinate is minimal. 
If either $z_{1}$ or $w_{1}$ is zero, we define $C_{1}=C$. In the case 
where neither $z_{1}$ nor $w_{1}$ is zero, we define a polygon $C_{1}$ 
by the following procedure. We first construct a 
polygon $\square_{E_{1}}$ 
from $\square_{C}$ by connecting $O$ and $P_{1}$. 
We put $P_{2}=(z_{1},w_{2})=l(C,(1,1))\cap l(C,(1,0))$ and 
$P_{3}=(z_{3},0)=l(C,(1,1))\cap l(C,(0,1))$ (see Fig. \ref{fig0}). 
%
\begin{figure}[h]\hspace{-5mm}
\begin{picture}(450,87)
\put(165,77){\line(-1,0){100}}
\put(164.9,77.2){\line(0,-1){77}}
\multiput(99.1,75.7)(0,0.4){5}{\line(-1,0){15.2}}
\multiput(99,75.9)(0.1,0.38){5}{\line(4,-1){25.3}}
\multiput(124,69.7)(0.23,0.35){5}{\line(3,-2){18}}
\multiput(113.7,87)(0.3,0.3){3}{\line(1,-1){59}}
\multiput(142,57.7)(0.34,0.29){5}{\line(4,-5){22}}
\multiput(163.8,16.1)(0.4,0){5}{\line(0,1){15}}
\multiput(165.5,15.8)(-0.29,0.29){5}{\line(-1,-1){15}}
\multiput(84.2,75.9)(-0.14,0.38){5}{\line(-3,-1){18}}
\put(79,80.2){$O$}
\put(168,10){$P_{1}$}
\put(176.3,23.3){$l(C,(1,1))$}
\put(98,30){$\square_{C}$}
\put(348,77){\line(-1,0){65.3}}
\put(265,77){\line(-1,0){17.3}}
\put(347.9,77.2){\line(0,-1){47}}
\put(347.9,17){\line(0,-1){16.7}}
\multiput(297,87)(0.3,0.3){3}{\line(1,-1){59}}
\multiput(348.6,15.5)(-0.29,0.29){5}{\line(-1,-1){14.7}}
\multiput(267.2,75.9)(-0.14,0.38){5}{\line(-3,-1){18}}
\multiput(266.3,76.3)(0.23,0.34){5}{\line(4,-3){81.5}}
\linethickness{0.3mm}
\qbezier[6](267,77)(274,77)(281,77)
\qbezier[9](281,77)(294,73.75)(307,70.5)
\qbezier[8](307,70.5)(316.5,64.25)(326,58)
\qbezier[12](326,58)(337,44.25)(348,30.5)
\qbezier[6](348,30.5)(348,23.5)(348,16.5)
\put(307.5,77){\circle*{3}}
\put(348,37){\circle*{3}}
\put(263,80.2){$O$}
\put(351,11){$P_{1}$}
\put(351,38){$P_{2}$}
\put(309,80.5){$P_{3}$}
\put(274,30){$\square_{E_{1}}$}
\end{picture}
\caption{}\label{fig0}
\end{figure}
%
If $z_{1}+w_{1}\geq 0$ 
(resp. $z_{1}+w_{1}<0$), we denote by $\square_{C_{1}}$ the 
convex hull of $\square_{E_{1}}\cup \{P_{2}\}$ (resp. 
$\square_{E_{1}}\cup \{P_{3}\}$). By definition, $C_{0}$ clearly 
satisfies the property (i). 
Let us show the inequality $C_{1}^{2}\leq C^{2}$. Since this is 
obvious if $C_{1}=C$, we consider the case where neither 
$z_{1}$ nor $w_{1}$ is equal to zero. If $z_{1}+w_{1}\geq 0$, we 
can take a non-negative integer $a$ such 
that the lattice point $P_{4}=P_{2}+a(-1,1)$ is contained 
in $l(C,(1,1))\cap \square_{C}$. We denote 
by $\square_{E_{2}}$ the convex hull of $\square_{E_{1}}\cup \{P_{4}\}$. 
Since $E_{2}^{\,2}\leq C^{2}$, it is sufficient to verify 
$C_{1}^{2}\leq E_{2}^{\,2}$. Note that the difference 
between $C_{1}^{2}$ and $E_{2}^{\,2}$ is caused only by the two 
sides $P_{2}P_{1}$ and $P_{4}P_{1}$. Hence we obtain 
$$C_{1}^{2}-E_{2}^{\,2}=(w_{2}-w_{1})z_{1}
-(w_{2}+a-w_{1})(z_{1}-a)+a(-w_{2}-a)=-a(z_{1}+w_{1})\leq 0.$$
Similarly we can show $C_{1}^{2}\leq C^{2}$ in the case 
where $z_{1}+w_{1}<0$. The shape of upper right corner 
of $\square_{C_{1}}$ is one of the three types as in Fig. \ref{fig01}, 
%
\begin{figure}[h]\hspace{-10mm}
\begin{picture}(450,45)
\put(162.3,41.3){\line(-1,0){38}}
\put(162.3,41.3){\line(0,-1){27.7}}
\multiput(162.8,40)(-0.12,0.4){5}{\line(-5,-1){36}}
\put(164.5,38.5){$Q_{1}$}
\put(129,0){$z_{1}=0$}
\put(240.8,41.3){\line(-1,0){38}}
\put(240.7,41.3){\line(0,-1){27.7}}
\multiput(241.1,41.1)(-0.37,0.12){5}{\line(-1,-4){7}}
\put(243.5,38.5){$Q_{1}$}
\put(205,0){$w_{1}=0$}
\put(335.2,41){\line(-1,0){38}}
\put(335.2,41){\line(0,-1){27.7}}
\multiput(305.2,40)(0.12,0.38){5}{\line(3,-1){30}}
\put(337.7,38.5){$Q_{1}$}
\put(285,0){$z_{1}>0,\,w_{1}<0$}
\end{picture}
\caption{}\label{fig01}
\end{figure}
%
where we put $Q_{1}=l(C,(0,1))\cap l(C,(1,0))$. By adapting a similar 
operation to other three corners of $\square_{C_{1}}$, we construct a 
polygon $\square_{C_{2}}$. It is obvious that $C_{2}$ satisfies the 
properties (i) and (ii). Besides, every vertices 
of $\square_{C_{2}}$ is on one of the four lines $l(C_{2},(0,\pm 1))$, 
$l(C_{2},(\pm 1,0))$. We put $Q_{4}=l(C,(-1,0))\cap l(C,(0,1))$. Let 
us show that if $l(C_{2},(0,1))$ contains two distinct vertices 
of $\square_{C_{2}}$ and one 
of them is $Q_{1}$ (resp. $Q_{4}$), then the other one is $Q_{4}$ 
(resp. $Q_{1}$). Assume that $\square_{C_{2}}$ contains $Q_{1}$ but 
not $Q_{4}$. Considering the construction method of $\square_{C_{2}}$, 
we deduce that $Q_{1}$ is contained in $\square_{C_{1}}$ also (see 
Fig. \ref{fig23}). 
%
\begin{figure}[h]\hspace{-9mm}
\begin{picture}(450,58)
\setlength\unitlength{0.8pt}
\put(84,65){\line(1,0){80}}
\put(84,0){\line(0,1){65}}
\put(164,0){\line(0,1){65}}
\multiput(83.5,12.5)(0.37,0){5}{\line(0,1){15}}
\multiput(84.5,26.7)(-0.28,0.28){5}{\line(1,1){15}}
\multiput(99.5,41.6)(-0.2,0.29){5}{\line(2,1){25}}
\multiput(124.5,54)(-0.1,0.34){5}{\line(4,1){40}}
\multiput(83.7,12.5)(0.23,0.28){5}{\line(3,-2){15}}
\put(117,22){$\square_{C_{1}}$}
\put(167,62){$Q_{1}$}
\put(63.5,7){$Q_{0}$}
\put(63.5,62){$Q_{4}$}
\put(192,22){or}
\put(239,65){\line(1,0){80}}
\put(239,0){\line(0,1){65}}
\put(319,0){\line(0,1){65}}
\multiput(238.5,12)(0.37,0){5}{\line(0,1){12}}
\multiput(239.5,23)(-0.28,0.28){5}{\line(3,5){9.5}}
\multiput(249,39)(-0.26,0.28){5}{\line(5,4){17}}
\multiput(266,52.4)(-0.1,0.34){5}{\line(3,1){35}}
\multiput(300.5,64)(0,0.37){5}{\line(1,0){19}}
\multiput(238.7,12)(0.23,0.28){5}{\line(3,-2){14}}
\put(273,22){$\square_{C_{1}}$}
\put(322,62){$Q_{1}$}
\put(218.5,62){$Q_{4}$}
\put(218.5,7){$Q_{0}$}
\put(347,22){or}
\put(394,65){\line(1,0){80}}
\put(394,0){\line(0,1){65}}
\put(474,0){\line(0,1){65}}
\multiput(394,64)(0,0.37){5}{\line(1,0){80}}
\multiput(393.5,65.5)(0.37,0){5}{\line(0,-1){14}}
\multiput(473,65.5)(0.32,-0.1){5}{\line(-1,-3){6}}
\put(426,33){$\square_{C_{1}}$}
\put(477,62){$Q_{1}$}
\put(373.5,62){$Q_{4}$}
\end{picture}
\caption{}\label{fig23}
\end{figure}
%
We denote by $Q_{0}$ the vertex of $\square_{C_{1}}$ on 
$l(C_{1},(-1,0))$ whose $w$-coordinate is minimal. Since the slant of 
the segment $Q_{1}Q_{0}$ is at most one, in the first two cases, 
$\square_{C_{2}}$ has only one vertex $Q_{1}$ on $l(C_{2},(0,1))$. 
In the third case, it is obvious that $\square_{C_{2}}$ has two 
vertices $Q_{1}$ and $Q_{4}$ on $l(C_{2},(0,1))$. 
We next consider the case where $\square_{C_{2}}$ contains $Q_{4}$ but 
not $Q_{1}$. Note that $Q_{1}$ is not contained in $\square_{C_{1}}$. 
Hence we obtain the four possibilities for the upper shape 
of $\square_{C_{1}}$ as in Fig. \ref{figl}. 
%
\begin{figure}[h]\hspace{-5mm}
\begin{picture}(450,34)
\setlength\unitlength{0.8pt}
\put(37,25){\line(1,0){72}}
\put(37,0){\line(0,1){25}}
\put(109,0){\line(0,1){25}}
\multiput(79,24.1)(0.24,0.47){4}{\line(5,-2){30}}
\multiput(80,24)(0,0.37){5}{\line(-1,0){15}}
\multiput(48,18)(-0.2,0.39){5}{\line(3,1){18}}
\put(65,2){$\square_{C_{1}}$}
\put(112,22){$Q_{1}$}
\put(18.5,22){$Q_{4}$}
\put(132,5){or}
\put(168,25){\line(1,0){72}}
\put(168,0){\line(0,1){25}}
\put(240,0){\line(0,1){25}}
\multiput(203,24.1)(0.14,0.38){5}{\line(3,-1){37}}
\multiput(181,12.6)(-0.21,0.42){5}{\line(2,1){23}}
\put(196,2){$\square_{C_{1}}$}
\put(243,22){$Q_{1}$}
\put(149.5,22){$Q_{4}$}
\put(262,5){or}
\put(295,25){\line(1,0){72}}
\put(295,0){\line(0,1){25}}
\put(367,0){\line(0,1){25}}
\multiput(294.5,24)(0,0.37){5}{\line(1,0){19}}
\multiput(313,24)(0.12,0.49){4}{\line(5,-1){54}}
\put(323,2){$\square_{C_{1}}$}
\put(371,22){$Q_{1}$}
\put(277.5,22){$Q_{4}$}
\put(390,5){or}
\put(424,25){\line(1,0){72}}
\put(424,0){\line(0,1){25}}
\put(496,0){\line(0,1){25}}
\multiput(423.7,23.9)(0.12,0.49){4}{\line(6,-1){72.5}}
\put(452,2){$\square_{C_{1}}$}
\put(499,22){$Q_{1}$}
\put(405.5,22){$Q_{4}$}
\end{picture}
\caption{}\label{figl}
\end{figure}
%
By the assumption $Q_{4}\in \square_{C_{2}}$, the first two cases can be 
excluded. The third case does not occur. Indeed, since $Q_{4}$ must be 
contained in $\square_{C}$ in this case, we have $z_{1}+w_{1}\geq 0$. 
It follows that $\square_{C_{1}}$ does not have vertices on 
$l(C_{1},(0,1))$ except for $Q_{4}$. Thus only the last case remains, 
in which $\square_{C_{2}}$ has one vertex $Q_{4}$ on $l(C_{2},(0,1))$. 
Similarly, with respect to the points $Q_{2}=l(C,(1,0))\cap 
l(C,(0,-1))$ and $Q_{3}=l(C,(0,-1))\cap l(C,(-1,0))$, we can show that 
if $l(C_{2},(0,-1))$ contains two distinct vertices of $\square_{C_{2}}$ 
and one of them is $Q_{2}$ (resp. $Q_{3}$), then the other one 
is $Q_{3}$ (resp. $Q_{2}$). \\
(a) Consider the case where $Q_{1}$ and $Q_{4}$ are contained in 
$\square_{C_{2}}$. In this case, $\square_{C_{2}}$ has at most 
six vertices (see Fig. \ref{fig03}). 
\begin{figure}[h]\hspace{-10mm}
\begin{picture}(450,45)
\multiput(197.9,41)(0,0.35){5}{\line(1,0){55.9}}
\multiput(198.1,41)(0.35,0){5}{\line(0,-1){20}}
\multiput(252.2,41)(0.35,0){5}{\line(0,-1){12.8}}
\multiput(198.2,21)(0.2,0.4){5}{\line(1,-2){10.5}}
\multiput(220,-0.1)(0,0.33){7}{\line(6,5){33.5}}
\multiput(208.5,-0.1)(0,0.39){5}{\line(1,0){11.6}}
\put(215,22){$\square_{C_{2}}$}
\put(256,39.5){$Q_{1}$}
\put(182,39.5){$Q_{4}$}
%
\end{picture}
\caption{}\label{fig03}
\end{figure}
%
The right and left vertical sides and the lower horizontal 
one may not exist. We denote by $Q$ the vertex of $\square_{C_{2}}$ on 
$l(C_{2},(0,-1))$ whose $z$-coordinate is minimal. Then 
we can finish the proof by defining $\square_{C_{0}}$ as a 
triangle $Q_{1}Q_{4}Q$. Indeed, a simple computation shows 
that $C_{0}$ satisfies the property (i). \\
(b) An argument similar to that in (a) goes through for the case where 
$Q_{2}$ and $Q_{3}$ are contained in $\square_{C_{2}}$. \\
(c) We put $L_{1}=l(C_{2},(0,1))\cap \square_{C_{2}}$ and 
$L_{2}=l(C_{2},(0,-1))\cap \square_{C_{2}}$, and 
consider the case where $L_{1}$ and $L_{2}$ are not the segments 
$Q_{1}Q_{4}$ and $Q_{2}Q_{3}$, respectively. In this case, the 
polygon $\square_{C_{2}}$ is as in Fig. \ref{fig04} (I). 
%
\begin{figure}[h]\hspace{-7mm}
\begin{picture}(450,128)
\put(41,30){\line(1,0){90}}
\put(41,110){\line(1,0){90}}
\put(41,30){\line(0,1){80}}
\put(131,30){\line(0,1){80}}
\multiput(84,109.2)(0,0.38){5}{\line(1,0){21}}
\multiput(105,110.7)(-0.2,-0.3){6}{\line(3,-2){27}}
\multiput(130,93)(0.38,0){5}{\line(0,-1){26}}
\multiput(131.5,66.5)(-0.23,0.25){6}{\line(-5,-4){46}}
\multiput(61,29.5)(0,0.38){5}{\line(1,0){24}}
\multiput(40.4,71)(0.34,0.2){5}{\line(1,-2){21}}
\multiput(40.2,71)(0.38,0){5}{\line(0,1){13}}
\multiput(40.7,83)(-0.15,0.39){5}{\line(5,3){44}}
\put(97.5,113.5){$O$}
\put(134,108){$Q_{1}$}
\put(134,31){$Q_{2}$}
\put(134,89){$R_{1}=(u_{1},v_{1})$}
\put(134,63){$R_{2}=(u_{1},v_{2})$}
\put(80,17.2){$R_{3}=(u_{3},-q)$}
\put(76,67){$\square_{C_{2}}$}
\linethickness{0.25mm}
\qbezier[30](41,30)(21,70)(41,110)
\qbezier[35](41,110)(88,142)(131,110)
\color{white}
\put(31,72){\circle*{10}}
\put(86,126){\circle*{8}}
\color{black}
\put(27,70){$q$}
\put(83,124){$q'$}
\put(79,0){(I)}
%
\linethickness{0.4pt}
\put(254,110){\line(-1,0){35}}
\multiput(254.7,109)(-0.15,0.48){4}{\line(-4,-1){35}}
\multiput(253.1,110.4)(0.5,0){4}{\line(0,-1){44}}
\put(211,16){$L_{1}=\{Q_{1}\}$}
\put(225,70){$\square_{C_{2}}$}
\put(326,30){\line(-1,0){35}}
\multiput(321,29.8)(0.1,0.44){4}{\line(-5,1){35}}
\multiput(320.1,29.6)(0.5,0){4}{\line(0,1){40}}
\put(280,16){$L_{2}=\{Q_{2}\}$}
\put(291,54){$\square_{C_{2}}$}
\put(293,0){(I\hspace{-0.3mm}I)}
\put(393,30){\line(-1,0){15}}
\put(393,110){\line(-1,0){26}}
\put(393,30){\line(0,1){80}}
\multiput(392.2,70)(0.4,0.22){5}{\line(-2,3){26.6}}
\multiput(393.5,70.3)(-0.4,0.18){5}{\line(-1,-3){13.6}}
\multiput(367,108.9)(0,0.5){4}{\line(-1,0){15}}
\multiput(379.3,29.6)(0,0.5){4}{\line(-1,0){20}}
\put(352,16){$R_{1}=R_{2}$}
\put(360,70){$\square_{C_{2}}$}
\end{picture}
\caption{}\label{fig04}
\end{figure}
%
We construct a polygon $\square_{C_{3}}$ by the following procedure. 
If one of the equalities $L_{1}=\{Q_{1}\}$, $L_{2}=\{Q_{2}\}$ and 
$R_{1}=R_{2}$ holds, we define $\square_{C_{3}}=\square_{C_{2}}$ 
(see Fig. \ref{fig04} (I\hspace{-0.3mm}I)). Assume 
that $L_{1}\neq \{Q_{1}\}$, $L_{2}\neq \{Q_{2}\}$ and $R_{1}\neq R_{2}$. 
If $u_{1}+v_{1}\leq 0$ (resp. $u_{1}+v_{1}>0$ and $u_{1}-u_{3}\leq 
v_{2}+q$), we construct $\square_{C_{3}}$ from $\square_{C_{2}}$ by 
connecting $O$ and $R_{2}$ (resp. $R_{1}$ and $R_{3}$) as in 
Fig. \ref{fig26}. 
%
\begin{figure}[h]\hspace{-10mm}
\begin{picture}(450,95)
\setlength\unitlength{0.95pt}
\put(204,12){\line(-1,0){70}}
\put(204,92){\line(-1,0){70}}
\put(204,11.8){\line(0,1){26}}
\put(204,61){\line(0,1){31.3}}
\multiput(165,91.2)(0,0.38){5}{\line(1,0){21}}
\multiput(186.3,92.8)(-0.32,-0.2){5}{\line(1,-3){18.2}}
\multiput(204.7,38.7)(-0.23,0.25){6}{\line(-5,-4){34.2}}
\multiput(170.7,11.5)(0,0.38){5}{\line(-1,0){25}}
\multiput(145.5,11.7)(0.2,0.49){4}{\line(-2,1){11}}
\multiput(165.7,91.2)(-0.15,0.39){5}{\line(-5,-3){29.6}}
\linethickness{0.3mm}
\qbezier[13](186,92)(195,76.5)(204,61)
\qbezier[8](204,39)(204,50)(204,61)
\put(180.5,95.5){$O$}
\put(207.5,57.5){$R_{1}$}
\put(207.5,34.5){$R_{2}$}
\put(164,0){$R_{3}$}
\put(152,47){$\square_{C_{3}}$}
\linethickness{0.4pt}
\put(329,12){\line(-1,0){70}}
\put(329,92){\line(-1,0){70}}
\put(329,11.8){\line(0,1){33}}
\put(329,62){\line(0,1){30}}
\multiput(272,91.2)(0,0.38){5}{\line(1,0){20}}
\multiput(292,92.8)(-0.25,-0.33){5}{\line(5,-4){37.5}}
\multiput(329.6,62.1)(-0.23,0.25){6}{\line(-1,-2){25.3}}
\multiput(303.8,11.5)(0,0.38){5}{\line(-1,0){25.5}}
\multiput(278,11.7)(0.2,0.49){4}{\line(-2,1){17}}
\multiput(273,91.2)(-0.15,0.39){5}{\line(-5,-3){11.6}}
\linethickness{0.3mm}
\qbezier[15](303.8,11.7)(316.4,28.1)(329,44.5)
\qbezier[6](329,44.5)(329,52.3)(329,61.1)
\put(288,95.5){$O$}
\put(332.5,59.3){$R_{1}$}
\put(332.5,41){$R_{2}$}
\put(295,0){$R_{3}$}
\put(277,47){$\square_{C_{3}}$}
\end{picture}
\caption{}\label{fig26}
\end{figure}
%
On the other hand, in the case where $u_{1}+v_{1}>0$ and 
$u_{1}-u_{3}>v_{2}+q$, 
we put $l(-1,-1)\cap l(-1,0)=(u_{1}-q',v_{4})$. 
If $v_{4}\leq v_{2}$ (resp. $v_{4}>v_{2}$), 
we define $R_{0}=R_{2}$ (resp. $R_{0}=(u_{1},v_{4})$), and 
construct $\square_{C_{3}}$ from $\square_{C_{2}}$ by 
connecting $O$, $R_{0}$ and $R_{3}$. In each case, one can easily 
verify that $C_{3}$ satisfies the properties (i) and (ii). By applying 
a similar operation to the opposite side of $\square_{C_{3}}$, we can 
construct a lattice polygon $\square_{C_{4}}$ satisfying the 
properties (i) and (ii). There exist four types of the shape 
of $\square_{C_{4}}$ as in Fig. \ref{fig25}, 
%
\begin{figure}[h]\hspace{-10mm}
\begin{picture}(450,42)
\multiput(94.4,40)(0,0.35){5}{\line(1,0){21.5}}
\multiput(73.3,22.1)(-0.3,0.3){4}{\line(6,5){22}}
\multiput(115,40.5)(0.35,0.23){4}{\line(4,-5){13}}
\multiput(120.8,0.1)(-0.36,0.2){4}{\line(1,3){8.2}}
\multiput(106.5,0)(0,0.35){5}{\line(1,0){14.3}}
\multiput(72.6,22.7)(0.22,0.33){4}{\line(3,-2){34}}
\put(96,20){$\square_{C_{4}}$}
\multiput(175.7,40)(0,0.35){5}{\line(1,0){20.4}}
\multiput(153.3,22.2)(-0.31,0.31){4}{\line(5,4){23.1}}
\multiput(195,40.5)(0.4,0.2){4}{\line(1,-2){12.8}}
\multiput(207.7,15.2)(0.35,0){5}{\line(0,-1){15.3}}
\multiput(208.5,0)(0.1,0.29){5}{\line(-5,2){56}}
\put(177,20){$\square_{C_{4}}$}
\multiput(292.5,22.2)(0.35,0){5}{\line(0,1){19.3}}
\multiput(293.9,40.2)(-0.1,0.45){4}{\line(-5,-1){56}}
\multiput(237.5,0)(0.35,0){5}{\line(0,1){30.4}}
\multiput(238,0)(-0.1,0.29){5}{\line(5,2){56}}
\put(254,20){$\square_{C_{4}}$}
\multiput(320.5,26.7)(-0.1,0.32){5}{\line(4,1){54}}
\multiput(373.5,0)(0.35,0){5}{\line(0,1){41.6}}
\multiput(320,27)(0.2,0.4){4}{\line(2,-1){53.8}}
\put(349,20){$\square_{C_{4}}$}
\end{picture}
\caption{}\label{fig25}
\end{figure}
%
where we ignore the reflection about $z$-axis or $w$-axis 
or both. We remark that, in the first two cases, 
vertical or horizontal sides may not exist. 
If $\square_{C_{4}}$ is a triangle or a square, we can finish 
the proof by putting $C_{0}=C_{4}$. If $\square_{C_{4}}$ has more than 
four vertices, we construct a polygon $\square_{C_{5}}$ by the 
following procedure (see Fig. \ref{fig24}). 
%
\begin{figure}[h]\hspace{-10mm}
\begin{picture}(450,55)
\put(183,45){\line(1,0){80}}
\put(183,0){\line(0,1){45}}
\put(263,0){\line(0,1){45}}
\multiput(222.2,44.2)(0,0.38){5}{\line(1,0){16.8}}
\multiput(239,45.7)(-0.2,-0.3){6}{\line(6,-5){24.8}}
\multiput(263.8,24.5)(-0.29,0.29){5}{\line(-5,-6){20}}
\multiput(182.4,13)(0.34,0.2){5}{\line(1,-2){6}}
\multiput(182.7,12)(-0.15,0.39){5}{\line(5,4){40}}
\put(215,15){$\square_{C_{4}}$}
\put(218,49.5){$O$}
\put(235.5,50){$S_{1}$}
\put(266,21){$S_{2}=(s_{2},t_{2})$}
\put(105,9.5){$S_{3}=(s_{2}-q',t_{3})$}
\end{picture}
\caption{}\label{fig24}
\end{figure}
%
If $O=S_{1}$, we define $\square_{C_{5}}=\square_{C_{4}}$. Assume 
that $O\neq S_{1}$. If $t_{2}\geq t_{3}$ (resp. $t_{2}<t_{3}$), 
we construct $\square_{C_{5}}$ from $\square_{C_{4}}$ by 
connecting $O$ and $S_{2}$ (resp. $S_{1}$ and $S_{3}$). Then an easy 
computation shows that $\square_{C_{5}}$ satisfies the properties (i) 
and (ii). By applying a similar operation to the lower side 
of $\square_{C_{5}}$, we obtain the desired 
lattice polygon $\square_{C_{0}}$. 
\end{proof}
\begin{remark}\label{animaruzonbi}
As it is apparent from the construction method, the 
equality $C^{2}=C_{0}^{2}$ holds if and only 
if $\square_{C}=\square_{C_{0}}$. 
\end{remark}
Using Lemma \ref{ikkakuusagi}, we can find the lower bound of the 
self-intersection number of $C$. 
\begin{proposition}\label{34}
Let $C$ be a nef curve on $S$, and $C_{0}$ a curve as in 
Lemma \ref{ikkakuusagi}. Then the inequality 
$C_{0}^{2}\geq \frac{3}{4}q^{2}$ $\big($in particular, 
$C^{2}\geq \frac{3}{4}q^{2}\big)$ holds. 
\end{proposition}
\begin{proof}
Recall that when we construct $C_{0}$, we divided 
the situation into the three cases (a), (b) and (c) in the proof of 
Lemma \ref{ikkakuusagi}. In cases (a) and (b), an easy 
computation shows that $C_{0}^2=qq'\geq q^{2}$. Let us consider 
case (c), that is, we assume that both $l(C_{0},(0,1))$ and 
$l(C_{0}(0,-1))$ contain only one vertex. We keep the notation 
$Q_{1},\ldots ,Q_{4}$ in the proof of Lemma \ref{ikkakuusagi}. 
Then the polygon $\square_{C_{0}}$ is drawn as in Fig. \ref{fig05}, 
%
\begin{figure}[h]\hspace{-10mm}
\begin{picture}(450,99)
\setlength\unitlength{0.92pt}
\put(187,13){\line(1,0){108}}
\put(187,97){\line(1,0){108}}
\put(187,13){\line(0,1){84}}
\put(295,13){\line(0,1){84}}
\multiput(187.5,48.3)(-0.3,0.3){5}{\line(5,4){60.4}}
\multiput(246,96.2)(0.16,0.3){6}{\line(4,-3){49}}
\multiput(259.5,12.2)(-0.28,0.28){6}{\line(3,4){36}}
\multiput(186.5,48.8)(0.2,0.3){6}{\line(2,-1){73}}
\linethickness{0.25mm}
\qbezier[19](246,97)(270.5,114)(295,97)
\qbezier[20](295,61)(310,37)(295,13)
\qbezier[27](187,13)(223.5,-5)(260,13)
\qbezier[20](187,97)(172,73)(187,49)
\color{white}
\put(251,100){\circle*{3}}
\put(272,106){\circle*{8}}
\put(225,4){\circle*{7}}
\color{black}
\put(241,101.5){$P_{1}$}
\put(298.5,58){$P_{2}$}
\put(254,0){$P_{3}$}
\put(171,46){$P_{4}$}
\put(171,99){$Q_{4}$}
\put(298.5,8){$Q_{2}$}
\put(269,103.5){$a$}
\put(305.5,32){$b$}
\put(222,0){$c$}
\put(171.5,70){$e$}
\put(236,50){$\square_{C_{0}}$}
\end{picture}
\caption{}\label{fig05}
\end{figure}
%
where we define 
\begin{itemize}
\setlength{\itemsep}{-2pt}
\item[$P_{1}$\,:\!]the vertex of 
$\square_{C_{0}}$ on $l(C_{0},(0,1))$, 
\item[$P_{2}$\,:\!]$\left\{
\begin{array}{lr}
\!\mbox{the vertex on $l(C_{0},(1,0))\setminus \{Q_{1}\}$}
 &\hspace{-36mm} \mbox{($\square_{C_{0}}$ has two vertices 
on $l(C_{0},(1,0))\setminus \{Q_{2}\}$),}\\
\!\mbox{the vertex on $l(C_{0},(1,0))$ whose $w$-coordinate is 
maximal}
 & ({\rm otherwise}),
\end{array}\right.$
\item[$P_{3}$\,:\!]the vertex of 
$\square_{C_{0}}$ on $l(C_{0},(0,-1))$, 
\item[$P_{4}$\,:\!]$\left\{
\begin{array}{lr}
\!\mbox{the vertex on $l(C_{0},(-1,0))\setminus \{Q_{3}\}$}\\
 &\hspace{-41.2mm} \mbox{($\square_{C_{0}}$ has two vertices 
on $l(C_{0},(-1,0))\setminus \{Q_{4}\}$),}\\
\!\mbox{the vertex on $l(C_{0},(-1,0))$ whose $w$-coordinate is minimal}
 &\! ({\rm otherwise}).
\end{array}\right.$
\end{itemize}
Note that $(b,e)\neq (q,0),(0,q)$. By computing, 
we obtain the formula 
$$C_{0}^{2}=qq'+(a+c-q')(b+e-q).$$
Without loss of 
generality, we can assume that $b+e\geq q$ and $(q'-a)b\geq (q'-c)e$. \\
(i) In the case where $a+e\geq q'$, we have $b+c\geq q'$. Indeed, 
if not, the inequality $eb\geq (q'-a)b\geq (q'-c)e$ gives that 
$e=0$ and $b=q$, a contradiction. 
Since the line $l(C_{0},(1,-1))$ (resp. $l(C_{0},(-1,1))$) passes 
through the point $P_{3}$ (resp. $P_{1}$), the 
condition $d(C_{0},(1,-1))\geq q'$ implies that $a+q-(q'-c)\geq q'$. 
Hence we have $C_{0}^{2}\geq qq'+(q'-q)(b+e-q)\geq qq'$. \\
(ii) Assume that $a+e<q'$ and $b+c\leq q'$. 
Since the line $l(C_{0},(1,-1))$ (resp. $l(C_{0},(-1,1))$) passes 
through the point $P_{2}$ (resp. $P_{4}$), we have 
$q'-e+q-b\geq q'$. It follows that $b+e=q$ and $C_{0}^{2}=qq'$. \\
(iii) Assume that $a+e<q'$ and $b+c>q'$. Since 
the line $l(C_{0},(1,-1))$ (resp. $l(C_{0},(-1,1))$) passes through 
the point $P_{3}$ (resp. $P_{4}$), we have $q'-e+q-(q'-c)\geq q'$. 
Hence we have 
\begin{eqnarray}\label{furogga}
C_{0}^{2}\geq qq'+(a+e-q)(b+e-q)
=qq'+\left(e+\frac{a+b-2q}{2}\right)^2-
\left(\frac{a-b}{2}\right)^{2},
\end{eqnarray}
where the equality holds if and only if $c-e=q'-q$ or $b+e=q$. 
If $a\geq b$ and $b+e=q$, we have $C^{2}=qq'$. On the other 
hand, if $a\geq b$ and $b+e>q$, we have $a+e>q$ and 
$C_{0}^{2}>qq'$ by the first inequality of 
(\ref{furogga}). Lastly, if $a<b$, we have $0<b-a\leq q$ and 
\begin{eqnarray}\label{ariahan}
C_{0}^{2}\geq qq'-\frac{q^{2}}{4}\geq \frac{3}{4}q^{2}.
\end{eqnarray}
%
\end{proof}
Proposition \ref{34} yields the following interesting corollary, though, 
which has no direct relation to the subject of this paper. 
\begin{corollary}\label{zoma}
Let $S$ be a compact smooth toric surface. For a $k$-gonal nef curve $C$ 
on $S$, 
the inequality $C^{2}\geq \frac{3}{4}k^{2}$ holds. 
\end{corollary}
Considering an irredundant embedding of $C$, we can obtain a more precise 
lower bound for $C^{2}$ when $q$ takes a small value. 
Here, the `irredundancy' has the following meaning: 
If there exists a $T$-invariant divisor $D_{i}$ on $S$ such that 
$D_{i}^{2}=-1$ and $C.D_{i}\leq 1$, then by blowing it down, we can 
embed $C$ in another compact smooth toric surface. By carrying out 
such operation repeatedly, we obtain an embedding satisfying 
the following condition. 
\begin{definition}\label{rebe}
Let $C$ be a smooth curve on $S$. 
The pair $(S,C)$ (or simply the curve $C$) is said to 
be {\it relatively minimal} if $C.D_{i}\geq 2$ for any $T$-invariant 
divisor $D_{i}$ on $S$ with self-intersection number $-1$. 
\end{definition}
In the remaining part of this section, 
we set $O=l(C,(0,1))\cap l(C,(1,0))$. 
Note that it is equal to the point $l(C_{0},(0,1))\cap l(C_{0},(1,0))$. 
\begin{proposition}\label{k3}
Let $C$ be a curve as in Theorem \ref{mthm}, and assume that $(S,C)$ is 
relatively minimal. If $q=2$ $($resp. $3)$, 
then $C^{2}\geq 12$ 
$($resp. $18)$. 
\end{proposition}
\begin{proof}
We shall prove only for the case $q=3$. Considering the 
relative minimality of $C$ and the reflection of $\square_{C}$ about 
$z$-axis, we can assume that $O$ is contained in $\square_{C}$. 
Moreover, we see that the right shape of $\square_{C}$ must be a segment 
connecting $O$ and $(-3m,-3)$, where $m$ is a non-negative integer. 
If $m\neq 0$, then by the condition $d(C,(1,-1))\geq q'$, we see 
that $(-q',0)\in \square_{C}$ and the left shape of $\square_{C}$ is 
a segment connecting $(-q',0)$ and $(-q'+3n,-3)$, where $n$ is 
non-negative integer. In the case $m=0$, by the relative 
minimality of $C$, the 
left shape of $\square_{C}$ is a segment connecting $(-q',-3)$ and 
$(-q'+3l,0)$, or $(-q',0)$ and $(-q'+3l,-3)$, where $l$ is non-negative 
integer. Consequently, without 
loss of generality, we can assume that $\square_{C}$ is a trapezium 
(possibly a triangle) as in Fig. \ref{fig36}, 
%
\begin{figure}[h]\hspace{-10mm}
\begin{picture}(450,45)
\linethickness{0.1mm}
\multiput(170,13)(0,10){4}{\line(1,0){110}}
\multiput(170,13)(10,0){12}{\line(0,1){30}}
\linethickness{0.4pt}
\multiput(169.7,42.5)(0,0.42){3}{\line(1,0){110}}
\multiput(220.3,12.6)(-0.18,0.36){3}{\line(2,1){60}}
\multiput(199.8,12.7)(0,0.42){3}{\line(1,0){20}}
\multiput(169.8,42.5)(0.23,0.23){4}{\line(1,-1){30}}
\color{white}
\put(210,25){\line(0,1){10.3}}
\put(220,24.7){\line(0,1){5.9}}
\put(204,33){\line(1,0){10.5}}
\color{black}
\put(200,13){\circle*{4}}
\put(220,13){\circle*{4}}
\put(205,26){$\square_{C}$}
\put(282,40){$O$}
\put(184,2){\vector(3,2){12}}
\put(118,0.5){$(-q'+3n,-3)$}
\put(205,0.5){$(-3m,-3)$}
\end{picture}
\caption{}\label{fig36}
\end{figure}
%
where the inequalities $m\geq n\geq 0$ hold. 
The inequality $C^{2}\geq 18$ is obvious if $q'\geq 6$. 
If $q'=5$, then there exist two possibilities 
$(m,n)=(0,0),(1,0)$. 
In the cases $q'=3,4$, since $C$ is not isomorphic to 
a plane curve, $(m,n)$ must be $(0,0)$. In each case, we obtain 
$C^{2}\geq 18$ by computing. 
\end{proof}
\begin{proposition}\label{k4}
Let $C$ be a curve as in Theorem \ref{mthm}, and 
assume that $(S,C)$ is relatively minimal. If $q=4$ and $C^{2}\leq 16$, 
the shape of $\square_{C}$ is one of the six types in 
Fig. \ref{fig35}, provided that we ignore congruence relations. 
%
\begin{figure}[h]\hspace{-9mm}
\begin{picture}(450,55)
\linethickness{0.1mm}
\multiput(53,14)(0,10){5}{\line(1,0){40}}
\multiput(53,14)(10,0){5}{\line(0,1){40}}
\multiput(53.1,33.6)(-0.32,0.38){3}{\line(2,1){40}}
\multiput(52.5,33.8)(0.33,0.33){3}{\line(1,-1){20.2}}
\multiput(73.1,13.8)(-0.36,0.18){3}{\line(1,2){20}}
\multiput(113,14)(0,10){5}{\line(1,0){40}}
\multiput(113,14)(10,0){5}{\line(0,1){40}}
\multiput(112.5,53.2)(0,0.22){6}{\line(1,0){40.9}}
\multiput(112.5,53.5)(0.3,0.15){4}{\line(1,-2){20}}
\multiput(153.3,53.5)(-0.3,0.15){4}{\line(-1,-2){20}}
\multiput(173,14)(0,10){5}{\line(1,0){40}}
\multiput(173,14)(10,0){5}{\line(0,1){40}}
\multiput(173.1,33.6)(-0.32,0.38){3}{\line(2,1){40}}
\multiput(193.1,13.8)(-0.36,0.18){3}{\line(1,2){20}}
\multiput(172.6,34)(0.22,0){6}{\line(0,-1){20.3}}
\multiput(172.5,13.7)(0,0.22){6}{\line(1,0){20.3}}
\multiput(233,14)(0,10){5}{\line(1,0){40}}
\multiput(233,14)(10,0){5}{\line(0,1){40}}
\multiput(233.2,33.6)(-0.32,0.38){3}{\line(2,1){40.1}}
\multiput(232.6,34)(0.22,0){6}{\line(0,-1){20}}
\multiput(272.3,54.1)(0.22,0){6}{\line(0,-1){20.2}}
\multiput(233.2,13.6)(-0.32,0.38){3}{\line(2,1){40.1}}
\multiput(293,14)(0,10){5}{\line(1,0){40}}
\multiput(293,14)(10,0){5}{\line(0,1){40}}
\multiput(293.2,33.6)(-0.32,0.38){3}{\line(2,1){40.1}}
\multiput(292.4,33.9)(0.33,0.33){3}{\line(1,-1){20.3}}
\multiput(332.3,54.1)(0.22,0){6}{\line(0,-1){20.2}}
\multiput(313.2,13.6)(-0.33,0.33){3}{\line(1,1){20.1}}
\multiput(353,14)(0,10){5}{\line(1,0){40}}
\multiput(353,14)(10,0){5}{\line(0,1){40}}
\multiput(353.2,33.5)(-0.25,0.25){4}{\line(1,1){20.1}}
\multiput(353,33.3)(0.25,0.25){4}{\line(1,-1){19.8}}
\multiput(392.7,33.5)(0.25,0.25){4}{\line(-1,1){20.2}}
\multiput(393.2,33.5)(-0.25,0.25){4}{\line(-1,-1){20.1}}
\put(54,0){$C^{2}=12$}
\put(114,0){$C^{2}=16$}
\put(174,0){$C^{2}=16$}
\put(234,0){$C^{2}=16$}
\put(294,0){$C^{2}=16$}
\put(354,0){$C^{2}=16$}
\end{picture}
\caption{}\label{fig35}
\end{figure}
\end{proposition}
\vspace{-4mm}
\begin{proof}
We take a curve $C_{0}$ as in Lemma \ref{ikkakuusagi}. 
Recall the three cases (a), (b) and (c) in the proof of 
Lemma \ref{ikkakuusagi}. In case (a) (that is, the upper side 
of $\square_{C_{0}}$ is a horizontal line of length $q'$), 
an easy computation gives $C_{0}^{2}\geq 16$. Suppose that $C^{2}=16$. 
Then, since $C^{2}=C_{0}^{2}=16$, we obtain the five possibilities 
for the shape of $\square_{C_{0}}$ as in Fig. \ref{fig12}. 
%
\begin{figure}[h]\hspace{-5.5mm}
\begin{picture}(450,55)
\linethickness{0.1mm}
\multiput(78,15)(0,10){5}{\line(1,0){40}}
\multiput(78,15)(10,0){5}{\line(0,1){40}}
\put(24,1){$(a,b,c,e)=(0,4,0,0)$}
\multiput(148,15)(0,10){5}{\line(1,0){40}}
\multiput(148,15)(10,0){5}{\line(0,1){40}}
\put(148,1){$(0,4,1,0)$}
\multiput(213,15)(0,10){5}{\line(1,0){40}}
\multiput(213,15)(10,0){5}{\line(0,1){40}}
\put(213,1){$(0,4,2,0)$}
\multiput(278,15)(0,10){5}{\line(1,0){40}}
\multiput(278,15)(10,0){5}{\line(0,1){40}}
\put(278,1){$(0,4,3,0)$}
\multiput(343,15)(0,10){5}{\line(1,0){40}}
\multiput(343,15)(10,0){5}{\line(0,1){40}}
\put(343,1){$(0,4,4,0)$}
\linethickness{0.4pt}
\multiput(77.5,54.5)(0,0.33){4}{\line(1,0){41}}
\multiput(77.5,14.5)(0.33,0){4}{\line(0,1){41}}
\multiput(78.3,14.5)(-0.25,0.25){4}{\line(1,1){40.1}}
\multiput(147.5,54.5)(0,0.33){4}{\line(1,0){40.9}}
\multiput(147.4,55)(0.3,0.1){4}{\line(1,-4){10.1}}
\multiput(158,14.3)(-0.28,0.23){4}{\line(3,4){30.4}}
\multiput(212.5,54.5)(0,0.33){4}{\line(1,0){40.9}}
\multiput(212.5,54.5)(0.3,0.15){4}{\line(1,-2){20}}
\multiput(253.3,54.5)(-0.3,0.15){4}{\line(-1,-2){20}}
\multiput(277.5,54.5)(0,0.33){4}{\line(1,0){40.9}}
\multiput(277.8,54.3)(0.28,0.23){4}{\line(3,-4){30.1}}
\multiput(318.5,55)(-0.3,0.1){4}{\line(-1,-4){10.1}}
\multiput(342.5,54.5)(0,0.33){4}{\line(1,0){40.9}}
\multiput(342.6,54.6)(0.25,0.25){4}{\line(1,-1){40.1}}
\multiput(382.5,14.5)(0.33,0){4}{\line(0,1){41}}
\color{white}
\multiput(87.5,40)(0.3,0){3}{\line(0,1){11.8}}
\multiput(97.5,38)(0.3,0){3}{\line(0,1){6}}
\multiput(80,44.5)(0,0.3){3}{\line(1,0){17.7}}
\multiput(157.5,40)(0.3,0){3}{\line(0,1){10.2}}
\multiput(167.5,38)(0.3,0){3}{\line(0,1){7.5}}
\multiput(153,44.5)(0,0.3){3}{\line(1,0){17.7}}
\multiput(222.5,40)(0.3,0){3}{\line(0,1){10.2}}
\multiput(232.5,38)(0.3,0){3}{\line(0,1){7.5}}
\multiput(221,44.5)(0,0.3){3}{\line(1,0){17.7}}
\multiput(297.5,40)(0.3,0){3}{\line(0,1){10.2}}
\multiput(307.5,37.2)(0.3,0){3}{\line(0,1){9}}
\multiput(291,44.5)(0,0.3){3}{\line(1,0){18}}
\multiput(362.5,40)(0.3,0){3}{\line(0,1){10.2}}
\multiput(372.5,38)(0.3,0){3}{\line(0,1){8}}
\multiput(360,44.5)(0,0.3){3}{\line(1,0){17.6}}
\color{black}
\put(81,42){$\square_{C_{0}}$}
\put(154,41){$\square_{C_{0}}$}
\put(222,41){$\square_{C_{0}}$}
\put(292,41){$\square_{C_{0}}$}
\put(361,41){$\square_{C_{0}}$}
\end{picture}
\caption{}\label{fig12}
\end{figure}
%
Note that $\square_{C}=\square_{C_{0}}$ by Remark \ref{animaruzonbi}. 
The cases of $(0,4,0,0)$ and $(0,4,4,0)$ are excluded by the 
assumption that $C$ is not isomorphic to a plane curve. On the other 
hand, the cases of $(0,4,1,0)$ and $(0,4,3,0)$ contradict the relative 
minimality. By a similar argument, one can show the lemma in case (b). 

Let us consider case (c). We follow the idea in the proof of 
Proposition \ref{34}. Namely, we divide the situation into the three 
cases (i), (ii) and (iii). In cases (i) and (ii), we have proved 
the inequality $C_{0}^{2}\geq 16$. Moreover, if equality holds, 
then $q'$ must be equal to four. In case (iii), by the 
inequality (\ref{ariahan}), we have $q'\leq 5$ if $C_{0}^{2}\leq 16$. 
Suppose that $q'=5$ and $C_{0}^{2}=16$. Then we deduce 
that $c-e=q'-q$, $e=(2q-a-b)/2$ and $a-b=-q$ by (\ref{furogga}), 
that is, $(a,b,c,e)=(0,4,3,2)$ (see Fig. \ref{fig16}). 
%
\begin{figure}[h]\hspace{-11mm}
\begin{picture}(450,41)
\linethickness{0.1mm}
\multiput(199,0)(0,10){5}{\line(1,0){50}}
\multiput(199,0)(10,0){6}{\line(0,1){40}}
\linethickness{0.4pt}
\multiput(198.8,19.5)(-0.15,0.4){3}{\line(5,2){50.3}}
\multiput(198.9,19.5)(0.25,0.35){3}{\line(3,-2){29.8}}
\multiput(249.4,39.8)(-0.36,0.19){3}{\line(-1,-2){20.2}}
\color{white}
\put(219,15.3){\line(0,1){10.3}}
\put(229,13.7){\line(0,1){7.5}}
\put(215,20){\line(1,0){17.2}}
\color{black}
\put(215.5,16.5){$\square_{C_{0}}$}
\put(251.5,37){$O$}
\end{picture}
\caption{}\label{fig16}
\end{figure}
%
Note that $\square_{C}=\square_{C_{0}}$ by 
Remark \ref{animaruzonbi}. This contradicts the relative minimality. 
Hence it is sufficient to consider the case $q'=4$. 

Let us examine the possibility of the shape of $\square_{C}$ satisfying 
$q=q'=4$ and $C^{2}\leq 16$. We denote by $P$ the vertex of 
$\square_{C}$ on $l(C,(0,1))$. Note that $P\neq (-1,0),(-3,0)$. Indeed, 
if $(-1,0)$ (resp. $(-3,0)$) is contained in $\square_{C}$, then 
also $(0,0)$ (resp. $(-4,0)$) is contained in $\square_{C}$ by the 
relative minimality, which contradicts to the assumption in (c). Since 
the case $P=(-4,0)$ is essentially equivalent to the case $P=O$, 
it is sufficient to consider the two cases $P=O,(-2,0)$. 
Assume that $P=O$. In this case, we can assume that none of 
$(-3,0)$, $(-4,0)$ and $(-4,-1)$ is contained in $\square_{C}$. 
Indeed, if not, $(-4,0)$ is contained in $\square_{C}$ by 
the relative minimality, a contradiction. 
We can take a unique integer $a$ with $-4\leq a\leq 0$ such 
that the line $l(C,(1,-1))$ passes through $(0,a)$. 
Since $d(C,(1,-1))\geq 4$, the cases $a=0,-1$ do not occur. 
If $a=-3$ or $-4$, then by the relative minimality, 
the point $(0,-4)$ must be contained in $\square_{C}$. Hence, by an 
argument similar to that in case (a), we see that $\square_{C}$ is 
a triangle with vertices $(0,0)$, $(0,-4)$ and $(-4,-2)$. In the 
case $a=-2$, since $d(C,(1,-1))\geq 4$, at least one of 
the points $(-2,0)$, $(-3,-1)$ and $(-4,-2)$ is contained in 
$\square_{C}$. 
%
%
We put $R_{1}=(-2,-4)$ and $R_{2}=(-4,-4)$. 
Assume that $R_{1},R_{2}\in \square_{C}$. Then, by computing, 
we obtain the three types of $\square_{C}$ satisfying $C^{2}\leq 16$ as 
in Fig. \ref{fig27}. 
%
\begin{figure}[h]\hspace{-11mm}
\begin{picture}(450,44)
\linethickness{0.1mm}
\multiput(143,0)(0,10){5}{\line(1,0){40}}
\multiput(143,0)(10,0){5}{\line(0,1){40}}
\multiput(143.1,19.6)(-0.32,0.38){3}{\line(2,1){40}}
\multiput(163.1,-0.2)(-0.36,0.18){3}{\line(1,2){20}}
\multiput(142.6,20)(0.22,0){6}{\line(0,-1){20.3}}
\multiput(142.5,-0.3)(0,0.22){6}{\line(1,0){20.3}}
\multiput(203,0)(0,10){5}{\line(1,0){40}}
\multiput(203,0)(10,0){5}{\line(0,1){40}}
\multiput(213.1,29.5)(-0.21,0.38){3}{\line(3,1){30}}
\multiput(223.1,-0.2)(-0.36,0.18){3}{\line(1,2){20}}
\multiput(203.4,0.1)(-0.38,0.21){3}{\line(1,3){9.9}}
\multiput(202.5,-0.3)(0,0.22){6}{\line(1,0){20.3}}
\multiput(263,0)(0,10){5}{\line(1,0){40}}
\multiput(263,0)(10,0){5}{\line(0,1){40}}
\multiput(263.3,0.2)(-0.36,0.18){3}{\line(1,2){19.8}}
\multiput(283.1,-0.2)(-0.36,0.18){3}{\line(1,2){20}}
\multiput(262.5,-0.3)(0,0.22){6}{\line(1,0){20.3}}
\multiput(282.5,39.2)(0,0.22){6}{\line(1,0){20.3}}
\color{white}
\put(163,15){\line(0,1){10.3}}
\put(154.5,20){\line(1,0){16}}
\put(223,14){\line(0,1){10.3}}
\put(213.2,20){\line(1,0){16}}
\put(283,15){\line(0,1){10.3}}
\put(274.5,20){\line(1,0){16}}
\multiput(275.2,22.5)(0.25,0){5}{\line(0,1){2}}
\color{black}
\put(154.9,17){$\square_{C}$}
\put(213.5,17){$\square_{C}$}
\put(275,17){$\square_{C}$}
\end{picture}
\caption{}\label{fig27}
\end{figure}
%
By the relative minimality, the second type is excluded. 
Assume that $R_{1}\notin \square_{C}$ and $R_{2}\in 
\square_{C}$. In this case, by the relative minimality, the lower side 
of $\square_{C}$ must be the segment connecting two points $(0,-2)$ and 
$R_{2}$. Then, by computing, we obtain the three types 
of $\square_{C}$ satisfying $C^{2}\leq 16$ as in Fig. \ref{fig17}. 
%
\begin{figure}[h]\hspace{-10mm}
\begin{picture}(450,42)
\linethickness{0.1mm}
\multiput(143,0)(0,10){5}{\line(1,0){40}}
\multiput(143,0)(10,0){5}{\line(0,1){40}}
\multiput(143.1,19.6)(-0.32,0.38){3}{\line(2,1){40}}
\multiput(143.3,-0.3)(-0.32,0.38){3}{\line(2,1){39.8}}
\multiput(142.6,20)(0.22,0){6}{\line(0,-1){20.3}}
\multiput(182.1,40)(0.22,0){6}{\line(0,-1){20.3}}
\color{white}
\put(159.8,20){\line(1,0){16.5}}
\put(163,17.7){\line(0,1){10.5}}
\put(173,16.4){\line(0,1){7.4}}
\color{black}
\put(160.5,19.3){$\square_{C}$}
\multiput(203,0)(0,10){5}{\line(1,0){40}}
\multiput(203,0)(10,0){5}{\line(0,1){40}}
\multiput(213.1,29.5)(-0.21,0.38){3}{\line(3,1){30}}
\multiput(203.4,0.1)(-0.38,0.21){3}{\line(1,3){9.9}}
\multiput(242.1,40)(0.22,0){6}{\line(0,-1){20.3}}
\multiput(203.3,-0.3)(-0.32,0.38){3}{\line(2,1){39.8}}
\color{white}
\put(217.5,20){\line(1,0){16.5}}
\put(223,18){\line(0,1){10.5}}
\put(233,18){\line(0,1){6}}
\color{black}
\put(218,19.5){$\square_{C}$}
\multiput(263,0)(0,10){5}{\line(1,0){40}}
\multiput(263,0)(10,0){5}{\line(0,1){40}}
\multiput(263.3,0.2)(-0.36,0.18){3}{\line(1,2){19.8}}
\multiput(282.5,39.2)(0,0.22){6}{\line(1,0){20.3}}
\multiput(302.1,40)(0.22,0){6}{\line(0,-1){20.3}}
\multiput(263.3,-0.3)(-0.32,0.38){3}{\line(2,1){39.8}}
\color{white}
\put(278.5,20){\line(1,0){16.3}}
\put(283,18){\line(0,1){10.5}}
\put(293,17.2){\line(0,1){7}}
\color{black}
\put(279,19.5){$\square_{C}$}
\end{picture}
\caption{}\label{fig17}
\end{figure}
%
By the relative minimality, the second type is excluded. Assume that 
$R_{1}\in \square_{C}$ and $R_{2}\notin \square_{C}$. 
The relative minimality implies that neither $(-3,-4)$ nor $(-4,-3)$ is 
contained in $\square_{C}$. 
%
%
%
%
Hence there exist the four possibilities for the shape 
of $\square_{C}$ as in Fig. \ref{fig28}. 
%
\begin{figure}[h]\hspace{-19.5mm}
\begin{picture}(450,42)
\linethickness{0.1mm}
\multiput(85,0)(0,10){5}{\line(1,0){40}}
\multiput(85,0)(10,0){5}{\line(0,1){40}}
\multiput(85.1,19.4)(-0.23,0.23){4}{\line(1,1){20}}
\multiput(105.3,-0.3)(-0.23,0.23){4}{\line(1,1){19.8}}
\multiput(85.1,19.2)(0.23,0.23){4}{\line(1,-1){19.7}}
\multiput(124.1,40)(0.22,0){6}{\line(0,-1){20.3}}
\multiput(104.6,39.2)(0,0.22){6}{\line(1,0){20.3}}
\color{white}
\put(101.3,20){\line(1,0){16.5}}
\put(105,17.6){\line(0,1){10.7}}
\put(115,16.3){\line(0,1){7.5}}
\color{black}
\put(102,19){$\square_{C}$}
\put(81.8,-2.1){\footnotesize $\times$}
\put(91.8,-2.1){\footnotesize $\times$}
\put(81.8,7.9){\footnotesize $\times$}
\put(121.8,-2.1){\footnotesize $\times$}
\put(111.8,-2.1){\footnotesize $\times$}
\put(121.8,7.9){\footnotesize $\times$}
\put(81.8,37.9){\footnotesize $\times$}
\put(91.8,37.9){\footnotesize $\times$}
\put(81.8,27.9){\footnotesize $\times$}
\multiput(145,0)(0,10){5}{\line(1,0){40}}
\multiput(145,0)(10,0){5}{\line(0,1){40}}
\multiput(145.1,19.4)(-0.23,0.23){4}{\line(1,1){20}}
\multiput(145.1,19.2)(0.23,0.23){4}{\line(1,-1){19.7}}
\multiput(164.6,39.2)(0,0.22){6}{\line(1,0){20.3}}
\multiput(165.3,0.2)(-0.36,0.18){3}{\line(1,2){19.8}}
\color{white}
\put(156,20){\line(1,0){16.5}}
\put(165,17.6){\line(0,1){10.7}}
\color{black}
\put(156.5,19){$\square_{C}$}
\put(141.8,-2.1){\footnotesize $\times$}
\put(151.8,-2.1){\footnotesize $\times$}
\put(141.8,7.9){\footnotesize $\times$}
\put(181.8,-2.1){\footnotesize $\times$}
\put(171.8,-2.1){\footnotesize $\times$}
\put(181.8,7.9){\footnotesize $\times$}
\put(141.8,37.9){\footnotesize $\times$}
\put(151.8,37.9){\footnotesize $\times$}
\put(141.8,27.9){\footnotesize $\times$}
\multiput(205,0)(0,10){5}{\line(1,0){40}}
\multiput(205,0)(10,0){5}{\line(0,1){40}}
\multiput(205.1,19.5)(-0.32,0.38){3}{\line(2,1){39.8}}
\multiput(225.3,-0.3)(-0.23,0.23){4}{\line(1,1){19.8}}
\multiput(204.9,19.4)(0.23,0.23){4}{\line(1,-1){19.9}}
\multiput(244.1,40)(0.22,0){6}{\line(0,-1){20.3}}
\color{white}
\put(221.3,20){\line(1,0){17.3}}
\put(225,15.6){\line(0,1){10.5}}
\put(235,14){\line(0,1){7.6}}
\color{black}
\put(222,16.8){$\square_{C}$}
\put(201.8,-2.1){\footnotesize $\times$}
\put(211.8,-2.1){\footnotesize $\times$}
\put(201.8,7.9){\footnotesize $\times$}
\put(241.8,-2.1){\footnotesize $\times$}
\put(231.8,-2.1){\footnotesize $\times$}
\put(241.8,7.9){\footnotesize $\times$}
\put(201.8,37.9){\footnotesize $\times$}
\put(211.8,37.9){\footnotesize $\times$}
\put(201.8,27.9){\footnotesize $\times$}
\multiput(265,0)(0,10){5}{\line(1,0){40}}
\multiput(265,0)(10,0){5}{\line(0,1){40}}
\multiput(265.1,19.5)(-0.32,0.38){3}{\line(2,1){39.8}}
\multiput(285.3,-0.2)(-0.38,0.32){3}{\line(1,2){19.9}}
\multiput(264.9,19.4)(0.23,0.23){4}{\line(1,-1){19.8}}
\color{white}
\put(276.2,20){\line(1,0){17.2}}
\put(285,15.6){\line(0,1){10.5}}
\color{black}
\put(276.8,16.8){$\square_{C}$}
\put(261.8,-2.1){\footnotesize $\times$}
\put(271.8,-2.1){\footnotesize $\times$}
\put(261.8,7.9){\footnotesize $\times$}
\put(301.8,-2.1){\footnotesize $\times$}
\put(291.8,-2.1){\footnotesize $\times$}
\put(301.8,7.9){\footnotesize $\times$}
\put(261.8,37.9){\footnotesize $\times$}
\put(271.8,37.9){\footnotesize $\times$}
\put(261.8,27.9){\footnotesize $\times$}
\put(321,16){$\times$}
\put(332,16.5){:\,the point which is}
\put(337,2){not contained in $\square_{C}$}
\end{picture}
\caption{}\label{fig28}
\end{figure}
%
In the first case, we have $C^{2}=20$. 
Lastly, we consider the case $P=(-2,0)$. 
In order to avoid the duplication, we assume that $\square_{C}$ contains 
none of the four corner $O$, $(0,-4)$, $(-4,-4)$ and $(-4,0)$. Then 
only one possibility remains: $\square_{C}$ is a square with vertices 
$(-2,0)$, $(0,-2)$, $(-2,-4)$ and $(-4,-2)$. 
\end{proof}
\begin{proposition}\label{k5}
Let $C$ be a curve as in Theorem \ref{mthm}, and assume 
that $(S,C)$ is relatively minimal. If $q=5$, then $C^{2}\geq 25$. 
\end{proposition}
\begin{proof}
We take a curve $C_{0}$ as in Lemma \ref{ikkakuusagi}. By the same 
argument as that in the proof of Proposition \ref{k4}, we 
have $C_{0}^{2}\geq 25$ except for case (iii) in the proof of 
Proposition \ref{34}. Note that $q'\leq 6$, 
$c-e\geq q'-5$ and $a<b$ if $C_{0}^{2}\leq 24$. 
In the case $q'=6$, we obtain the two 
possibilities $(a,b,c,e)=(0,5,3,2),(0,5,4,3)$ by computing. In both 
cases, we have $C_{0}^{2}=24$. On the other hand, 
the relative minimality of $C$ implies that $\square_{C}$ is not equal 
to $\square_{C_{0}}$, which means that $C^{2}>C_{0}^{2}$. 

We next consider the case $q'=5$. If none of $O$, 
$(0,-5)$, $(-5,-5)$ and $(-5,0)$ is contained in $\square_{C}$, then 
also the eight points $(-1,0)$, $(0,-1)$, $(0,-5)$, 
$(-1,-5)$, $(-4,-5)$, $(-5,-4)$, $(-5,-1)$ and $(-4,0)$ are not 
contained in $\square_{C}$. Hence the left (resp. right) shape 
of $\square_{C}$ is one of three types 
in Fig. \ref{fig02} (I) (resp. (I\hspace{-0.3mm}I)). By 
%
\begin{figure}[h]\hspace{13mm}
\begin{picture}(450,80)
\linethickness{0.1mm}
\multiput(15,29)(0,10){6}{\line(1,0){36}}
\multiput(15,29)(10,0){4}{\line(0,1){50}}
\linethickness{0.4pt}
\multiput(15.3,48.8)(-0.23,0.23){4}{\line(1,1){30}}
\multiput(14.5,48.8)(0.23,0.23){4}{\line(1,-1){20}}
\put(14.8,49.2){\circle*{1}}
\put(11.8,76.9){\footnotesize $\times$}
\put(11.8,66.9){\footnotesize $\times$}
\put(21.8,76.9){\footnotesize $\times$}
\put(11.8,26.9){\footnotesize $\times$}
\put(21.8,26.9){\footnotesize $\times$}
\put(11.8,36.9){\footnotesize $\times$}
\linethickness{0.1mm}
\multiput(68,29)(0,10){6}{\line(1,0){36}}
\multiput(68,29)(10,0){4}{\line(0,1){50}}
\linethickness{0.4pt}
\multiput(68.3,58.8)(-0.23,0.23){4}{\line(1,1){20}}
\multiput(67.5,58.8)(0.23,0.23){4}{\line(1,-1){30}}
\put(67.8,59.2){\circle*{1}}
\put(64.8,76.9){\footnotesize $\times$}
\put(64.8,66.9){\footnotesize $\times$}
\put(74.8,76.9){\footnotesize $\times$}
\put(64.8,26.9){\footnotesize $\times$}
\put(74.8,26.9){\footnotesize $\times$}
\put(64.8,36.9){\footnotesize $\times$}
\linethickness{0.1mm}
\multiput(121,29)(0,10){6}{\line(1,0){27}}
\multiput(121,29)(10,0){3}{\line(0,1){50}}
\linethickness{0.4pt}
\multiput(121.3,58.8)(-0.23,0.23){4}{\line(1,1){20}}
\multiput(120.5,48.8)(0.23,0.23){4}{\line(1,-1){20}}
\multiput(120.5,49)(0.23,0){5}{\line(0,1){10}}
\put(117.8,76.9){\footnotesize $\times$}
\put(117.8,66.9){\footnotesize $\times$}
\put(127.8,76.9){\footnotesize $\times$}
\put(117.8,26.9){\footnotesize $\times$}
\put(127.8,26.9){\footnotesize $\times$}
\put(117.8,36.9){\footnotesize $\times$}
\put(77,14){(I)}
\linethickness{0.1mm}
\multiput(174,29)(0,10){6}{\line(1,0){36}}
\multiput(180,29)(10,0){4}{\line(0,1){50}}
\linethickness{0.4pt}
\multiput(209.5,48.95)(0.23,0.23){4}{\line(-1,1){30}}
\multiput(210.3,48.8)(-0.23,0.23){4}{\line(-1,-1){20}}
\put(209.95,49.3){\circle*{1}}
\put(196.8,76.9){\footnotesize $\times$}
\put(206.8,66.9){\footnotesize $\times$}
\put(206.8,76.9){\footnotesize $\times$}
\put(196.8,26.9){\footnotesize $\times$}
\put(206.8,26.9){\footnotesize $\times$}
\put(206.8,36.9){\footnotesize $\times$}
\linethickness{0.1mm}
\multiput(227,29)(0,10){6}{\line(1,0){36}}
\multiput(233,29)(10,0){4}{\line(0,1){50}}
\linethickness{0.4pt}
\multiput(262.5,58.95)(0.23,0.23){4}{\line(-1,1){20}}
\multiput(263.3,58.8)(-0.23,0.23){4}{\line(-1,-1){30}}
\put(262.95,59.2){\circle*{1}}
\put(249.8,76.9){\footnotesize $\times$}
\put(259.8,66.9){\footnotesize $\times$}
\put(259.8,76.9){\footnotesize $\times$}
\put(249.8,26.9){\footnotesize $\times$}
\put(259.8,26.9){\footnotesize $\times$}
\put(259.8,36.9){\footnotesize $\times$}
\linethickness{0.1mm}
\multiput(280,29)(0,10){6}{\line(1,0){27}}
\multiput(287,29)(10,0){3}{\line(0,1){50}}
\linethickness{0.4pt}
\multiput(306.5,58.95)(0.23,0.23){4}{\line(-1,1){20}}
\multiput(307.3,48.8)(-0.23,0.23){4}{\line(-1,-1){20}}
\multiput(306.2,49)(0.23,0){5}{\line(0,1){10}}
\put(293.8,76.9){\footnotesize $\times$}
\put(303.8,66.9){\footnotesize $\times$}
\put(303.8,76.9){\footnotesize $\times$}
\put(293.8,26.9){\footnotesize $\times$}
\put(303.8,26.9){\footnotesize $\times$}
\put(303.8,36.9){\footnotesize $\times$}
\put(230,14){(I\hspace{-0.3mm}I)}
\put(145,0.5){$\times$ :\,the point which is not contained 
in $\square_{C}$}

\end{picture}
\caption{}\label{fig02}
\end{figure}
%
noting the condition $d(C,(1,\pm 1))\geq 5$, we 
have $C^{2}\geq 25$ in any case. Therefore, considering the reflection, 
we see that it is sufficient to verify our lemma under the assumption 
$O\in \square_{C}$. Since the inequality $C^{2}\geq 25$ is obvious 
if $(0,-5)$ or $(-5,0)$ is contained in $\square_{C}$, we assume 
that $(0,-5),(-5,0)\notin \square_{C}$. It follows that also the four 
points $(0,-4)$, $(-1,-5)$, $(-5,-1)$ and $(-4,0)$ are not contained in 
$\square_{C}$. We denote by $P$ the vertex of $\square_{C}$ on 
$l(C,(0,-1))$ whose $z$-coordinate is maximal. 
Let us consider the case $P=(-2,-5)$. Then we 
have the four possibilities as in Fig. \ref{fig10}. 
%
\begin{figure}[h]\hspace{15mm}
\begin{picture}(450,90)
\linethickness{0.1mm}
\multiput(16,40)(0,10){6}{\line(1,0){50}}
\multiput(16,40)(10,0){6}{\line(0,1){50}}
\multiput(93,40)(0,10){6}{\line(1,0){50}}
\multiput(93,40)(10,0){6}{\line(0,1){50}}
\multiput(170,40)(0,10){6}{\line(1,0){50}}
\multiput(170,40)(10,0){6}{\line(0,1){50}}
\multiput(247,40)(0,10){6}{\line(1,0){50}}
\multiput(247,40)(10,0){6}{\line(0,1){50}}
\put(24,25){(I)}
\put(69.5,89){$O$}
\put(111,25){(I\hspace{-0.3mm}I)}
\put(185.3,25){(I\hspace{-0.4mm}I\hspace{-0.4mm}I)}
\put(262,25){(I\hspace{-0.4mm}V)}
\put(46,40){\circle*{3}}
\put(66,90){\circle*{3}}
\put(16,70){\circle*{3}}
\put(123,40){\circle*{3}}
\put(143,90){\circle*{3}}
\put(93,60){\circle*{3}}
\put(200,40){\circle*{3}}
\put(220,90){\circle*{3}}
\put(170,50){\circle*{3}}
\put(277,40){\circle*{3}}
\put(297,90){\circle*{3}}
\put(247,40){\circle*{3}}
\put(52.8,37.9){\footnotesize $\times$}
\put(62.8,37.9){\footnotesize $\times$}
\put(62.8,47.9){\footnotesize $\times$}
\put(12.8,87.9){\footnotesize $\times$}
\put(12.8,77.9){\footnotesize $\times$}
\put(22.8,87.9){\footnotesize $\times$}
\put(129.8,37.9){\footnotesize $\times$}
\put(139.8,37.9){\footnotesize $\times$}
\put(139.8,47.9){\footnotesize $\times$}
\put(89.8,87.9){\footnotesize $\times$}
\put(89.8,77.9){\footnotesize $\times$}
\put(99.8,87.9){\footnotesize $\times$}
\put(206.8,37.9){\footnotesize $\times$}
\put(216.8,37.9){\footnotesize $\times$}
\put(216.8,47.9){\footnotesize $\times$}
\put(166.8,87.9){\footnotesize $\times$}
\put(166.8,77.9){\footnotesize $\times$}
\put(176.8,87.9){\footnotesize $\times$}
\put(283.8,37.9){\footnotesize $\times$}
\put(293.8,37.9){\footnotesize $\times$}
\put(293.8,47.9){\footnotesize $\times$}
\put(243.8,87.9){\footnotesize $\times$}
\put(243.8,77.9){\footnotesize $\times$}
\put(253.8,87.9){\footnotesize $\times$}
\put(40.5,28.5){$P$}
\put(141,13){\circle*{3}}
\put(147,11){:\,the point which is contained in $\square_{C}$}
\put(136,0){$\times$ :\,the point which is not contained in $\square_{C}$}
\end{picture}
\caption{}\label{fig10}
\end{figure}
%
In case (I), by the relative minimality, we see that there exist 
integers $m_{1}$ and $m_{2}$ with $-3\leq m_{1},m_{2}\leq -1$ such 
that $(0,m_{1})$ and $(m_{2},0)$ is contained in $\square_{C}$. An
easy computation gives $C^{2}\geq 25$ in any case. Consider 
case (I\hspace{-0.3mm}I). 
By the relative minimality, we see that either $(0,-1)$ or $(0,-3)$ is 
contained in $\square_{C}$, and likewise either $(-3,-5)$ or $(-5,-5)$ 
is contained in $\square_{C}$. Then it is obvious that the minimum value 
of $C^{2}$ is attained when the lower shape of $\square_{C}$ is the 
polygonal line connecting $O$, $(0,-1)$, $(-2,-5)$, $(-3,-5)$ 
and $(-5,-3)$. Then we see that 
$C^{2}$ achieves its minimum 25 when the upper shape of $\square_{C}$ is 
the polygonal line connecting $(-5,-3)$, $(-4,-2)$ and $O$. 
Consider cases (I\hspace{-0.4mm}I\hspace{-0.4mm}I) and 
(I\hspace{-0.4mm}V). We note that the point $(-5,-5)$ is contained 
in $\square_{C}$ in case (I\hspace{-0.4mm}I\hspace{-0.4mm}I) also. 
By the condition $d(C,(1,-1))\geq 5$, 
$\square_{C}$ has a lattice point in the domain $A$ in 
Fig. \ref{fig11}. 
%
\begin{figure}[h]\hspace{-11mm}
\begin{picture}(450,51)
\glay
\multiput(198.2,20)(0.48,0){63}{\line(0,1){30}}
\color{white}
\multiput(197.9,30)(-0.45,0){44}{\line(1,1){20}}
\multiput(198,20)(0.45,0){67}{\line(1,1){30}}
\color{black}
\linethickness{0.1mm}
\multiput(198,0)(0,10){6}{\line(1,0){50}}
\multiput(198,0)(10,0){6}{\line(0,1){50}}
\linethickness{0.4pt}
\multiput(198.1,29.95)(-0.1,0.1){2}{\line(1,1){19.8}}
\multiput(198.1,19.95)(-0.1,0.1){2}{\line(1,1){29.8}}
\put(228,0){\circle*{3}}
\put(248,50){\circle*{3}}
\put(198,0){\circle*{3}}
\put(251.5,47){$O$}
\put(179.5,23.5){$A$}
\put(190,27){\vector(1,0){10}}
\put(194.8,47.9){\footnotesize $\times$}
\put(194.8,37.9){\footnotesize $\times$}
\put(204.8,47.9){\footnotesize $\times$}
\put(234.8,-2.1){\footnotesize $\times$}
\put(244.8,-2.1){\footnotesize $\times$}
\put(244.8,7.9){\footnotesize $\times$}
%
%
\end{picture}
\caption{}\label{fig11}
\end{figure}
%
When $\square_{C}$ is a square with vertices $O$, $(-2,-5)$, $(-5,-5)$ 
and $Q$, the self-intersection number $C^{2}$ achieves its 
minimum 25, where $Q$ is either $(-5,-3)$, $(-4,-2)$, $(-3,-1)$ or 
$(-2,0)$. In the case $P=(-3,-5)$, we assume 
that $(-5,-2)$ is not contained in $\square_{C}$ in order to avoid the 
duplication. By the relative minimality and the 
condition $d(C,(1,-1))\geq 5$, we see that $l(C,(1,-1))$ passes 
through $P$ and $(-3,0)$ is contained in $\square_{C}$. Then, since 
the upper shape of $\square_{C}$ must be the polygonal line connecting 
$(-5,-4)$, $(-3,0)$ and $O$, we obtain $C^{2}\geq 25$. 
We next consider the case $P=(-4,-5)$. By the relative minimality and 
the condition $d(C,(1,-1))\geq 5$, we see 
that $(-5,-5)$ is contained in $\square_{C}$, and moreover, 
either $(0,-2)$ or $(0,-3)$ is contained in $\square_{C}$. Then, 
considering the reflection and the rotation, this case can be reduced 
to the case where $P=(-2,-5)$ or $(-3,-5)$. The same argument goes 
through for the case $P=(-5,-5)$. 
\end{proof}
%
%
%
\section{Proof of the main theorem}\label{doruido}

To prove Theorem \ref{mthm}, we first aim to show that any 
gonality pencil on $C$ can be extended to a morphism from $S$. 
Let us prove several lemmas needed later. Also in this section, we use 
the notion of coprime in the wide sense (see Definition \ref{youganmajin}). 
\begin{lemma}\label{romalia}
Let $C$ be a curve as in Theorem \ref{mthm}. Assume that $(S,C)$ is 
relatively minimal and $q\geq 3$. Let $V$ be an effective divisor on $S$. 
\vspace{-2mm}
\begin{itemize}
\setlength{\itemsep}{-2pt}
\item[{\rm (i)}]\hspace{-1mm}If $h^{0}(S,V)\geq 2$ and 
$h^{0}(S,V+K_{S})\geq 1$, then $C.V\geq q+2$. 
\item[{\rm (ii)}]\hspace{-1mm}If $h^{0}(S,V)\geq 2$ and 
$h^{0}(S,V+K_{S})\geq 2$, then $C.V\geq q+3$. 
\item[{\rm (iii)}]\hspace{-1mm}If $h^{0}(S,V)\geq 3$ and 
$h^{0}(S,V+K_{S})\geq 1$, then $C.V\geq q+3$. 
\end{itemize}
\end{lemma}
\begin{proof}
We write $V=\sum_{i=1}^{d}n_{i}D_{i}$ with non-negative integer 
coefficients. We denote by $\sigma (D_{d_{0}})$ the cone 
in $\Delta_{S}$ whose primitive element is $(0,-1)$. \\
(i) By assumption, we can assume that the origin $O$ is contained in 
$\square_{V}$ and there exists another lattice point $P=(z,w)$ contained 
in the interior of $\square_{V}$. Without loss of generality, we can 
assume that $z\geq 0$, $w\leq 0$ and $(z,-w)=1$. 
We denote by $A_{1}$ the domain drawn in Fig. \ref{fig06} (I). 
%
\begin{figure}[h]\hspace{1mm}
\begin{picture}(450,130)
\linethickness{0.7mm}
\glay
\multiput(16.7,13)(2.02,0){36}{\line(0,1){109.2}}
\color{white}
\multiput(15,33.5)(3.3,4.4){21}{\circle*{10}}
\multiput(15,47)(3.3,4.4){18}{\circle*{10}}
\linethickness{14mm}
\multiput(15,73)(0,32){2}{\line(1,0){20}}
\put(34,101){\line(1,0){23}}
\put(34,113){\line(1,0){32}}
\linethickness{0.4pt}
\color{black}
\multiput(50.1,13)(0.52,0){4}{\line(0,1){115}}
\multiput(11,27)(0.4,-0.3){4}{\line(3,4){73}}
\put(78,115){\circle*{6}}
\put(51,78.7){\circle*{7}}
\put(84,110){$(-w,z)$}
\put(35,77){$O$}
\put(57,51){$A_{1}$}
\put(44.5,0){(I)}
\linethickness{0.7mm}
\glay
\multiput(166.5,13)(2,0){35}{\line(0,1){109}}
\color{white}
\multiput(165.5,20)(-0.2,0){3}{\line(1,2){30}}
\multiput(162,24)(2.2,4.4){14}{\circle*{10}}
\multiput(165,45)(2.2,4.4){9}{\circle*{10}}
\multiput(165,60)(2.2,4.4){5}{\circle*{10}}
\multiput(197,80)(0,0.2){3}{\line(1,1){40}}
\multiput(197,87)(3,3){13}{\circle*{10}}
\multiput(189,90)(3,3){12}{\circle*{10}}
\linethickness{14mm}
\multiput(165,90)(0,15){2}{\line(1,0){26}}
\put(189,115){\line(1,0){26}}
\linethickness{0.4pt}
\color{black}
\multiput(190.1,13)(0.52,0){4}{\line(0,1){101}}
\multiput(163,43)(0.4,-0.3){4}{\line(3,4){60}}
\linethickness{0.7mm}
\qbezier[12](190.5,77)(210.5,97)(230.5,117)
\qbezier[15](190.5,77)(176,48)(161.5,19)
\put(218,115){\circle*{6}}
\put(191,78.7){\circle*{7}}
\put(211,97){\circle*{6}}
\put(169.2,34){\circle*{6}}
\put(181,118){$(-w,z)$}
\put(105,32){$(-1,-m-1)$}
\put(202,51){$A_{2}$}
\put(217,91){$(1,m)$}
\put(182.5,0){(I\hspace{-0.3mm}I)}
\linethickness{0.7mm}
\glay
\multiput(270,13)(2,0){25}{\line(0,1){63}}
\color{white}
\multiput(270,29)(-0.48,0){3}{\line(1,1){49}}
\multiput(268,34)(3.2,3.2){15}{\circle*{10}}
\multiput(268,46)(3.2,3.2){11}{\circle*{10}}
\multiput(291.5,12)(2.6,5.2){12}{\circle*{10}}
\linethickness{8mm}
\put(268,65){\line(1,0){26}}
\put(297,24){\line(1,0){23}}
\put(308,44){\line(1,0){10}}
\linethickness{0.4pt}
\color{black}
\multiput(315.1,13)(0.52,0){4}{\line(0,1){115}}
\multiput(315,78)(0.32,-0.2){5}{\line(-1,-2){31.8}}
\linethickness{0.7mm}
\qbezier[16](315.5,78)(290.5,53)(265.5,28)
\put(294,34){\circle*{6}}
\put(316,78.7){\circle*{7}}
\put(294,56){\circle*{6}}
\put(251,61){$(-1,-m)$}
\color{white}
\multiput(315,26)(0.4,0){5}{\line(0,1){15}}
\color{black}
\put(300,30){$(-1,-m-1)$}
\put(270,20){$A_{3}$}
\put(306,0){(I\hspace{-0.3mm}I\hspace{-0.3mm}I)}
\end{picture}
\caption{}\label{fig06}
\end{figure}
%
Since $P$ is contained in the interior of $\square_{V}$, 
the inequality $x_{i}z+y_{i}w<n_{i}$ holds for any $(x_{i},y_{i})
\in A_{1}\cap {\rm Pr}(S)$. 
We thus obtain 
%
\begin{eqnarray}\label{ishisu}
\begin{array}{ll}
C.V &\!\!\! \displaystyle =
\sum_{i=1}^{d}n_{i}C.D_{i}
\geq \sum_{\sigma(D_{i})\subset A_{1}}
(x_{i}z+y_{i}w+1)C.D_{i}\\[5mm]
 &\!\!\! \displaystyle =
d(C,(-w,z))+\sum_{\sigma(D_{i})\subset A_{1}}C.D_{i}\geq 
q+\sum_{\sigma(D_{i})\subset A_{1}}C.D_{i}.
\end{array}
\end{eqnarray}
Thus it is sufficient to verify $\sum_{\sigma(D_{i})\subset 
A_{1}}C.D_{i}\geq 2$. This inequality 
is true if there exists a cone $\sigma(D_{i})\subset A_{1}$ such that 
$D_{i}^{2}=-1$. Hence we suppose that there does not exist such a cone 
(we call this the `{\it nonexistence condition}') and 
$\sum_{\sigma(D_{i})\subset A_{1}}C.D_{i}=1$. We can take a 
cone $\sigma(D_{i_{0}})\subset A_{1}$ such that 
\begin{eqnarray*}
C.D_{i}=\left\{
\begin{array}{ll}
1 & (i=i_{0}),\\
0 & (i\neq i_{0},\,\sigma(D_{i})\subset A_{1}).
\end{array}
\right.
\end{eqnarray*}
If there exists only one cone $\sigma(D_{j})$ included in 
$A_{1}\setminus {\mathbb R}(-w,z)$, then 
$d(C,(x_{j-1},y_{j-1}))$ is equal to one, a contradiction. 
We thus have $N=\sharp \{\sigma(D_{i})\in \Delta_{S}\mid 
\sigma(D_{i})\subset A_{1}\setminus {\mathbb R}(-w,z)\}\geq 2$. 
Then, by the nonexistence condition, we deduce that 
neither $z$ nor $w$ is equal to zero. We denote by $m$ the maximal 
integer satisfying $z+mw\geq 0$. By the nonexistence condition, there 
does not exist a cone included in the domain $A_{2}$ except 
for $\sigma(D_{d_{0}})$ (see Fig. \ref{fig06} (I\hspace{-0.3mm}I)). 
On the other hand, 
since $N\geq 2$, there exists a cone $\sigma(D_{l})\subset A_{1}$ such 
that $x_{l}\neq 0$. Consider the case where $x_{l}$ is positive. 
Since $\sigma(D_{l})\subset A_{1}\setminus A_{2}$, 
we have $(1,m)\in {\rm Pr}(S)$. Thus, it follows from 
$D_{d_{0}}^{2}\neq -1$ that there does not exist a cone in 
the domain $A_{3}$ (see Fig. \ref{fig06} 
(I\hspace{-0.3mm}I\hspace{-0.3mm}I)). We deduce that 
$$\{(x_{j},y_{j})\in M(x_{i_{0}-1},y_{i_{0}-1})
\mid C.D_{j}\geq 1\,\}=\{(x_{i_{0}},y_{i_{0}})\},$$
which implies the contradiction $d(C,(x_{i_{0}-1},y_{i_{0}-1}))=1$. 
In the case where $x_{l}$ is negative, one can 
obtain a similar contradiction. \\
(ii) In this case, we can assume that $\square_{V}$ has two 
distinct lattice points $(0,0)$ and $(z,w)$ in its interior, 
where $z\geq 0$, $w\leq 0$ and $(z,-w)=1$. As we saw in (i), 
the inequality $\sum_{\sigma(D_{i})\subset A_{1}}n_{i}C.D_{i}\geq q+2$ 
holds. On the other hand, since $(0,0)$ is contained in the interior of 
$\square_{V}$, the 
coefficient $n_{i}$ is positive for every $T$-invariant 
divisor $D_{i}$ ($i=1,\ldots ,d$). 
If $C.D_{i}=0$ for any $\sigma(D_{i})\not \subset A_{1}$, 
then we have $d(C,(-w,z))=0$, a contradiction. We thus have 
$C.V\geq q+2+
\sum_{\sigma(D_{i})\not \subset A_{1}}n_{i}C.D_{i}\geq q+3$. \\
(iii) In this case, there exist three distinct lattice points $(0,0)$, 
$(z,w)$ and $(z',w')$ in $\square_{V}$, especially $(z,w)$ is contained 
in the interior of $\square_{V}$. We can assume that $z\geq 0$, 
$w\leq 0$, $(z,-w)=1$ and $(|z'|,|w'|)=1$. 
Suppose that $C.V=q+2$, and denote by $i_{1}$ (resp. $i_{2}$) the 
minimal (resp. maximal) integer in $\{i\mid \sigma(D_{i})\subset 
A_{1},\,C.D_{i}\geq 1\}$. By a computation similar to that in 
(\ref{ishisu}), we obtain 
$\sum_{\sigma(D_{i})\subset A_{1}}C.D_{i}=2$ and 
$\sum_{\sigma(D_{i})\not \subset A_{1}}n_{i}C.D_{i}=0$. It 
follows that $C.D_{i}=0$ for $\sigma(D_{i})\subset A_{1}$ except 
for $i=i_{1},\,i_{2}$. Moreover, we see 
that $n_{i_{1}}=n_{i_{2}}=C.D_{i_{1}}=C.D_{i_{2}}=1$ (resp. 
$n_{i_{1}}=1$ and $C.D_{i_{1}}=2$) if 
$i_{1}\neq i_{2}$ (resp. $i_{1}=i_{2}$). Let us consider the case where 
$zw'-wz'\geq 0$. 
Let $\sigma(D_{j})$ be a cone included in the domain $B_{1}$ drawn 
in Fig. \ref{fig07}. 
%
\begin{figure}[h]\hspace{-10mm}
\begin{picture}(450,107)
\linethickness{0.7mm}
\glay
\multiput(142,33)(2,0){22}{\line(0,1){73}}
\color{white}
\multiput(142,35.5)(-0.3,0.1){5}{\line(1,4){18}}
\multiput(138,42)(1.5,6){12}{\circle*{10}}
\multiput(138,77)(1.5,6){6}{\circle*{10}}
\multiput(143,32)(0.4,-0.1){5}{\line(3,4){42}}
\multiput(149.5,32)(2.7,3.6){15}{\circle*{10}}
\multiput(161,35)(3,4){9}{\circle*{10}}
\linethickness{9.5mm}
\put(165,45.8){\line(1,0){20}}
\color{black}
\linethickness{0.7mm}
\qbezier[17](141,30.5)(150.25,67.5)(159.5,104.5)
\qbezier[26](120.2,1.6)(152.6,44.8)(185,88)
\linethickness{0.4pt}
\multiput(140.1,1)(0.52,0){4}{\line(0,1){106}}
\put(154.7,85){\circle*{6}}
\put(176.5,76.5){\circle*{6}}
\put(141,30.5){\circle*{7}}
\color{white}
\multiput(139,82)(0.4,0){11}{\line(0,1){15}}
\color{black}
\put(181,69){$(-w,z)$}
\put(125,29){$O$}
\put(163.5,85){$B_{1}$}
\put(112,86){$(-w',z')$}
\linethickness{0.7mm}
\glay
\multiput(271,0)(2,0){28}{\line(0,1){107}}
\color{white}
\multiput(270,-13)(-0.37,0.08){5}{\line(1,4){30}}
\multiput(266.8,-3)(1.5,6){19}{\circle*{10}}
\multiput(266.5,32)(1.5,6){14}{\circle*{10}}
\linethickness{6mm}
\put(275,55){\line(0,1){52}}
\color{black}
\linethickness{0.7mm}
\qbezier[25](274,0.5)(286.9,53.15)(299.8,105.8)
\qbezier[26](258.6,1.5)(291.8,44.25)(325,87)
\linethickness{0.4pt}
\multiput(280.1,0)(0.52,0){4}{\line(0,1){107}}
\put(294.7,85){\circle*{6}}
\put(317,76.5){\circle*{6}}
\put(281,30.5){\circle*{7}}
\color{white}
\multiput(279,82)(0.4,0){11}{\line(0,1){15}}
\color{black}
\put(265,29){$O$}
\put(298.2,29.5){$B_{2}$}
\put(322,69){$(-w,z)$}
\put(252,86){$(-w',z')$}
\end{picture}
\caption{}\label{fig07}
\end{figure}
%
Since $(z',w')$ is contained in $L(V,(x_{j},y_{j}))$, we have the 
inequalities $n_{j}\geq x_{j}z'+y_{j}w'>0$. Hence $C.D_{j}=0$. 
Noting that $(z',w')\in L(V,(x_{i_{1}},y_{i_{1}}))\cap 
L(V,(x_{i_{2}},y_{i_{2}}))$, we have 
\begin{eqnarray*}
\begin{array}{ll}
 &\!\!\! d(C,(-w',z'))=\displaystyle 
\sum_{\sigma(D_{i})\subset B_{2}}(x_{i}z'+y_{i}w')C.D_{i}\\[5.5mm]
= &\!\!\! \left\{\begin{array}{ll}
(w'y_{i_{1}}+z'x_{i_{1}})C.D_{i_{1}}+
(w'y_{i_{2}}+z'x_{i_{2}})C.D_{i_{2}}\leq n_{i_{1}}C.D_{i_{1}}+
n_{i_{2}}C.D_{i_{2}}=2
 & (i_{1}\neq i_{2}),\\
(w'y_{i_{1}}+z'x_{i_{1}})C.D_{i_{1}}\leq n_{i_{1}}C.D_{i_{1}}=2
 & (i_{1}=i_{2}).
\end{array}\right.
\end{array}
\end{eqnarray*}
This contradicts the assumption $q\geq 3$. A similar argument can be 
carried out for the case where $zw'-wz'\leq 0$. 
\end{proof}
We are now in a position to show the extension of a gonality pencil. 
In the proof, the following Serrano's result plays an essential role. 
\begin{theorem}[\cite{ser}]\label{serrano}
Let $X$ be a smooth curve on a smooth surface $Y$, and 
$f:X\rightarrow {\mathbb P}^{1}$ a surjective morphism of degree $p$. 
Assume that $X^{2}>4p$. If there does not exist an effective 
divisor $V$ on $Y$ satisfying the following properties $({\tt a})$ and 
$({\tt b})$, then there exists a morphism from $Y$ to $\mathbb P^{1}$ 
whose restriction to X is $f$. 
\\[1.8mm]
$\begin{array}{ll}
\,{\rm ({\tt a})}\ 1\leq V^{2}<(X-V).V\leq p, & 
{\rm ({\tt b})}\ \displaystyle X^{2}\leq 
\frac{(p+V^{2})^{2}}{V^{2}}
\end{array}$.
\end{theorem}
\begin{proposition}\label{kandata}
Let $C$ be a curve as in Theorem \ref{mthm}, and $f$ a gonality 
pencil on $C$. If $(g,q)\neq (4,4),(5,4),(10,6)$, then there exists 
a morphism from $S$ to $\mathbb P^{1}$ whose restriction to $C$ is $f$. 
\end{proposition}
\begin{proof}
If $(S,C)$ is not relatively minimal, by the method explained before 
Definition \ref{rebe}, we can obtain an equivariant morphism $\psi$ from 
$S$ to another compact smooth toric surface $S'$ such that $(S',C)$ is 
relatively minimal. Clearly, for 
a morphism $\varphi$ from $S'$ to $\mathbb P^{1}$, the 
composite $(\varphi \circ \psi)|_{C}$ coincides with $\varphi |_{C}$. 
Hence, it is sufficient to consider the case where $(S,C)$ is 
relatively minimal. 

By the condition $g\geq 2$, we have $q\geq 2$. If $q=2$, 
then $k=2$ and $C$ has only one gonality pencil. Thus our lemma is 
obvious in this case. Let us consider the case where $q\geq 3$. 
Suppose that, for $C$, $S$ and $f$, there exists an effective 
divisor $V$ satisfying the two properties ({\tt a}) and 
({\tt b}) in Theorem \ref{serrano}. We put $s=V^{2}$. 
By the inequality $q\geq k$ and Proposition \ref{34}, we have 
$1\leq s<k$ and $\frac{\mbox{3}}{\mbox{4}}k^{2}\leq C^{2}\leq 
\frac{\mbox{$(k+s)^{2}$}}{\mbox{$s$}}$. 
It follows that 
\begin{eqnarray*}
\left\{
\begin{array}{ll}
s=1 & (k\geq 9),\\
1\leq s\leq 2 & (k=7,8),\\
1\leq s\leq 3 & (k=6),\\
1\leq s\leq 4 & (k=5),\\
1\leq s\leq 3 & (k=4),\\
1\leq s\leq 2 & (k=3),\\
s=1 & (k=2).
\end{array}
\right.
\end{eqnarray*}

We first consider the case where $s\leq 2$. By Riemann-Roch 
theorem, we have 
\begin{eqnarray*}
\hspace{-1mm}
\begin{array}{rl}
0 &\!\!\! \leq h^{1}(S,V+K_{S})
=h^{0}(S,V+K_{S})+h^{0}(S,-V)-\frac{1}{2}(V+K_{S}).V-\chi ({\cal O}_{S})\\
\frac{1}{2}V.K_{S} &\!\!\! \leq h^{0}(S,V+K_{S})-\frac{s}{2}-1,\\
\end{array}
\end{eqnarray*}
\vspace{-7mm}
\begin{eqnarray}\label{samayouyoroi}
\hspace{8mm}\begin{array}{rl}
0 &\!\!\! \leq h^{1}(S,V)
=h^{0}(S,V)+h^{0}(S,K_{S}-V)-\frac{1}{2}V.(V-K_{S})-\chi ({\cal O}_{S})
\\
 &\!\!\! \leq h^{0}(S,V)+h^{0}(S,V+K_{S})-s-2.
\end{array}
\end{eqnarray}
Since $h^{0}(S,V)\geq h^{0}(S,V+K_{S})$, by (\ref{samayouyoroi}), we 
obtain $2h^{0}(S,V)\geq s+2$. 
Assume that $s=1$. Then we have $h^{0}(S,V)\geq 2$ 
and $C.V\leq k+1\leq q+1$ by the property 
({\tt a}). Hence we have $h^{0}(S,V+K_{S})=0$ by 
Lemma \ref{romalia} (i). It follows from (\ref{samayouyoroi}) 
that $h^{0}(S,V)\geq s+2$. In the case $s=2$, 
since $h^{0}(S,V)\geq 2$ and $C.V\leq q+2$, we 
have $h^{0}(S,V+K_{S})\leq 1$ by Lemma \ref{romalia} (ii). 
We thus have $h^{0}(S,V)\geq 3$ by (\ref{samayouyoroi}), and 
$h^{0}(S,V+K_{S})=0$ by Lemma \ref{romalia} (iii). It follows from 
(\ref{samayouyoroi}) that $h^{0}(S,V)\geq s+2$. On the other hand, 
since
\begin{eqnarray*}
C.(V-C)=C.V-C^{2}\leq \left\{
\begin{array}{ll}
k+s-\frac{3}{4}q^{2}\leq 
q+2-\frac{3}{4}q^{2}<0 & (q\geq 3),\\
k+s-12<0 & (q=2)
\end{array}
\right.
\end{eqnarray*}
by Proposition \ref{34} and Proposition \ref{k3}, 
we obtain $h^{0}(S,V-C)=0$. Therefore, in the case where $s\leq 2$, 
the cohomology exact sequence 
$$0\rightarrow H^{0}(S,V-C)\rightarrow H^{0}(S,V)\rightarrow 
H^{0}(C,V|_{C})\rightarrow \cdots$$
gives the inequality $h^{0}(C,V|_{C})\geq s+2$. If we 
write $g_{l}^{r}=|V|_{C}|$, then $r\geq s+1$ and $l\leq k+s$. We obtain 
a net $g_{l-r+2}^{2}$ by subtracting $r-2$ general points of $C$ from 
it. Since $C$ is not isomorphic to a plane curve, $g_{l-r+2}^{2}$ is 
not very ample. Then we obtain a pencil $g_{l-r}^{1}$ such 
that $l-r\leq k+s-(s+1)=k-1$, a contradiction. 

Let us show that the cases where $(k,s)=(4,3),(5,3),(5,4)$ do not occur. 
Assume that $k=4$. If $q\geq 5$, then we 
have $C^{2}\geq \frac{3}{4}q^{2}>18$. If $q=4$, then we 
deduce $C^{2}\geq 17$ by the assumption $(g,q)\neq (4,4),(5,4)$ and 
Proposition \ref{k4}. We conclude that $s$ must be at most two 
by the property ({\tt b}). In the case $k=5$, since $q\geq 5$, we 
have $C^{2}\geq 25$ by Proposition \ref{k5} and Proposition \ref{34}. 
Hence $s$ is one in this case. 

The case $(k,s)=(6,3)$ remains to consider. Assume 
that $k=6$ and $s=3$. We take a curve $C_{0}$ as in 
Lemma \ref{ikkakuusagi}. By Proposition \ref{34}, 
$$27=\frac{3}{4}k^{2}\leq \frac{3}{4}q^{2}\leq C_{0}^{2}\leq C^{2}\leq 
\frac{(k+s)^{2}}{s}=27,$$
which yields $q=6$ and $C_{0}^{2}=C^{2}=27$. Hence we have 
$\square_{C}=\square_{C_{0}}$ by Remark \ref{animaruzonbi}. 
By the argument in the proof of 
Proposition \ref{34}, we see that $C_{0}$ is a curve of type (iii) in 
it. Then the inequality (\ref{ariahan}) gives $q'=6$. 
Moreover, by the inequality (\ref{furogga}), we deduce that 
$c=e=6-(a+b)/2$ and 
$a-b=-6$. We thus have $(a,b,c,e)=(0,6,3,3)$ and $g=10$ 
(see Fig. \ref{fig29}). 
%
\begin{figure}[h]\hspace{-11mm}
\begin{picture}(450,73)
\linethickness{0.1mm}
\multiput(193,14)(0,10){7}{\line(1,0){60}}
\multiput(193,14)(10,0){7}{\line(0,1){60}}
\linethickness{0.4pt}
\multiput(192.5,43.8)(0.2,0.3){4}{\line(1,-1){30.2}}
\multiput(223.4,13.7)(-0.35,0.1){4}{\line(1,2){29.9}}
\multiput(193.2,43.4)(-0.2,0.3){4}{\line(2,1){59.9}}
\put(223,14.9){\circle*{1}}
\put(203,1){$(0,6,3,3)$}
\color{white}
\put(213,39){\line(0,1){10}}
\put(223,37.8){\line(0,1){6}}
\put(212,44){\line(1,0){18}}
\color{black}
\put(213,40){$\square_{C}$}
\end{picture}
\caption{}\label{fig29}
\end{figure}
\end{proof}
%
%
%
%
\begin{lemma}\label{g10}
Let $C$ be a curve as in Theorem \ref{mthm}, and assume that $(S,C)$ is 
relatively minimal. If $q=6$ and $g=10$, then $\square_{C}$ is a 
triangle as in Fig. \ref{fig29}. 
\end{lemma}
\begin{proof}
%
%
%
%
%
%
In this proof, we often use the relative minimality of $C$ and the 
property $d(C,(1,-1))\geq q'$ without further mention. We denote 
by Int\,$\square_{C}$ the interior of $\square_{C}$, and 
by $l((a_{1},b_{1}),(a_{2},b_{2}))$ the segment connecting two 
points $(a_{1},b_{1})$ and $(a_{2},b_{2})$. 
We assume that the point $l(C,(0,1))\cap l(C,(1,0))$ is the 
origin $O$. Now we suppose that none of $O$, $(0,-6)$, $(-q',-6)$ and 
$(-q',0)$ is contained in $\square_{C}$. It follows that also the 
eight points 
$(-1,0)$, $(0,-1)$, $(0,-5)$, $(-1,-6)$, $(-q'+1,-6)$, $(-q',-5)$, 
$(-q',-1)$, $(-q'+1,0)$ are not contained in $\square_{C}$. Assume 
that $(0,-3)$ is contained in $\square_{C}$. We can take a lattice 
point $P\in l(-1,0)\cap \square_{C}$. We define $A$ as a domain 
surrounded by the four segments $l((0,-3),(-q'+2,-6))$, 
$l((-2,-6),P)$, $l(P,(-2,0))$ and $l((-q'+2,0),(0,-3))$. In any case, 
we see that there exist more than ten lattice points in the interior 
of $A$ (see Fig. \ref{fig30}). 
%
\begin{figure}[h]\hspace{-11mm}
\begin{picture}(450,63)
\glay
\multiput(86.4,9)(0,0.48){94}{\line(1,0){57.2}}
\color{white}
\multiput(86.3,41.2)(0,0.45){3}{\line(2,1){25.6}}
\multiput(113.5,54)(0,0.45){3}{\line(4,-3){30}}
\multiput(117,9)(0.2,-0.3){3}{\line(4,3){26.5}}
\multiput(86.5,37.9)(-0.24,-0.24){4}{\line(1,-1){29}}
\multiput(85,47)(3.6,1.8){7}{\circle*{12}}
\multiput(118,57)(4,-3){8}{\circle*{10}}
\multiput(128,57)(4,-3){4}{\circle*{14}}
\multiput(120,5)(4,3){7}{\circle*{10}}
\multiput(135,10)(4,3){2}{\circle*{14}}
\multiput(85,32)(3,-3){9}{\circle*{10}}
\multiput(83,20)(3,-3){5}{\circle*{15}}
\color{black}
\linethickness{0.1mm}
\multiput(85,0)(0,10){7}{\line(1,0){60}}
\multiput(85,0)(10,0){7}{\line(0,1){60}}
\linethickness{0.4pt}
\multiput(105.1,0)(-0.2,0.2){2}{\line(4,3){39.9}}
\multiput(84.8,40)(0.22,0.22){2}{\line(1,-1){40}}
\multiput(85.2,39.9)(-0.2,0.22){2}{\line(2,1){40}}
\multiput(105,59.9)(0.18,0.2){2}{\line(4,-3){40}}
\put(131.1,57.5){\footnotesize $\times$}
\put(141.1,57.5){\footnotesize $\times$}
\put(141.1,47.5){\footnotesize $\times$}
\put(141.1,7.5){\footnotesize $\times$}
\put(141.1,-2.5){\footnotesize $\times$}
\put(131.1,-2.5){\footnotesize $\times$}
\put(91.1,-2.5){\footnotesize $\times$}
\put(81.1,-2.5){\footnotesize $\times$}
\put(81.1,7.5){\footnotesize $\times$}
\put(81.1,47.5){\footnotesize $\times$}
\put(81.1,57.5){\footnotesize $\times$}
\put(91.1,57.5){\footnotesize $\times$}
\put(85,40){\circle*{3}}
\put(145,30){\circle*{3}}
\put(148.5,59){$O$}
\put(72.5,37){$P$}
\glay
\put(110,30){\line(1,0){10}}
\put(115,26){\line(0,1){11}}
\color{black}
\put(110,27){$A$}
\glay
\multiput(196.4,8.2)(0,0.48){93}{\line(1,0){57.2}}
\color{white}
\multiput(196.3,31.5)(-0.21,0.32){4}{\line(4,3){28}}
\multiput(225.9,52.5)(0.22,0.33){3}{\line(4,-3){27.8}}
\multiput(226,8)(0.22,-0.33){3}{\line(4,3){27.6}}
\multiput(196.4,28.4)(-0.22,-0.33){3}{\line(4,-3){27.6}}
\multiput(196.3,38)(4,3){7}{\circle*{10}}
\multiput(198,49)(4,3){2}{\circle*{12}}
\multiput(232,54)(4,-3){7}{\circle*{10}}
\multiput(248,52)(5,-4){2}{\circle*{12}}
\multiput(232,6)(4,3){7}{\circle*{10}}
\multiput(248,9)(5,4){2}{\circle*{12}}
\multiput(195,24)(4,-3){7}{\circle*{10}}
\multiput(198,13)(5,-4){2}{\circle*{12}}
\color{black}
\linethickness{0.1mm}
\multiput(195,0)(0,10){7}{\line(1,0){60}}
\multiput(195,0)(10,0){7}{\line(0,1){60}}
\linethickness{0.4pt}
\multiput(215.1,0)(-0.2,0.2){2}{\line(4,3){39.9}}
\multiput(195.1,30)(-0.2,0.2){2}{\line(4,3){39.9}}
\multiput(195,29.9)(0.18,0.2){2}{\line(4,-3){40}}
\multiput(215,59.9)(0.18,0.2){2}{\line(4,-3){40}}
\put(241.1,57.5){\footnotesize $\times$}
\put(251.1,57.5){\footnotesize $\times$}
\put(251.1,47.5){\footnotesize $\times$}
\put(251.1,7.5){\footnotesize $\times$}
\put(251.1,-2.5){\footnotesize $\times$}
\put(241.1,-2.5){\footnotesize $\times$}
\put(201.1,-2.5){\footnotesize $\times$}
\put(191.1,-2.5){\footnotesize $\times$}
\put(191.1,7.5){\footnotesize $\times$}
\put(191.1,47.5){\footnotesize $\times$}
\put(191.1,57.5){\footnotesize $\times$}
\put(201.1,57.5){\footnotesize $\times$}
\put(195,30){\circle*{3}}
\put(255,30){\circle*{3}}
\put(258.5,59){$O$}
\put(182.5,27){$P$}
\glay
\put(220,30){\line(1,0){10}}
\put(225,26){\line(0,1){11}}
\color{black}
\put(220,27){$A$}
\glay
\multiput(306.3,6.6)(0,0.48){93}{\line(1,0){57.2}}
\color{white}
\multiput(306,21.7)(-0.22,0.22){3}{\line(1,1){29.3}}
\multiput(336,52.5)(0.2,0.31){4}{\line(4,-3){27.6}}
\multiput(334,6.6)(0.22,-0.33){3}{\line(4,3){29.4}}
\multiput(306.3,18.8)(-0.17,-0.34){4}{\line(2,-1){26}}
\multiput(306,28)(3.5,3.5){8}{\circle*{10}}
\multiput(310,40)(7,7){2}{\circle*{15}}
\put(308,50){\circle*{10}}
\multiput(343,53)(4,-3){6}{\circle*{10}}
\multiput(355,53)(6,-5){2}{\circle*{12}}
\multiput(340.2,5)(4,3){7}{\circle*{10}}
\multiput(355,8)(6,4){2}{\circle*{13}}
\multiput(306.3,12)(4.2,-2.1){5}{\circle*{12}}
\color{black}
\linethickness{0.1mm}
\multiput(305,0)(0,10){7}{\line(1,0){60}}
\multiput(305,0)(10,0){7}{\line(0,1){60}}
\linethickness{0.4pt}
\multiput(325.1,0)(-0.2,0.2){2}{\line(4,3){39.9}}
\multiput(305.1,20)(-0.22,0.22){2}{\line(1,1){39.9}}
\multiput(304.7,19.9)(0.2,0.22){2}{\line(2,-1){40}}
\multiput(325,59.9)(0.18,0.2){2}{\line(4,-3){40}}
\put(351.1,57.5){\footnotesize $\times$}
\put(361.1,57.5){\footnotesize $\times$}
\put(361.1,47.5){\footnotesize $\times$}
\put(361.1,7.5){\footnotesize $\times$}
\put(361.1,-2.5){\footnotesize $\times$}
\put(351.1,-2.5){\footnotesize $\times$}
\put(311.1,-2.5){\footnotesize $\times$}
\put(301.1,-2.5){\footnotesize $\times$}
\put(301.1,7.5){\footnotesize $\times$}
\put(301.1,47.5){\footnotesize $\times$}
\put(301.1,57.5){\footnotesize $\times$}
\put(311.1,57.5){\footnotesize $\times$}
\put(305,20){\circle*{3}}
\put(365,30){\circle*{3}}
\put(368.5,59){$O$}
\put(292.5,17){$P$}
\glay
\put(330,30){\line(1,0){10}}
\put(335,26){\line(0,1){11}}
\color{black}
\put(330,27){$A$}
\end{picture}
\caption{}\label{fig30}
\end{figure}
%
Since $\square_{C}$ includes $A$, we obtain $g\geq 11$. Next we 
assume that $(0,-4)$ is contained in $\square_{C}$. Similarly to the 
previous case, we obtain $g\geq 11$ except for the two cases 
where $q'=6$ and $P=(-6,-2),(-6,-4)$ (see Fig. \ref{fig31}). 
%
\begin{figure}[h]\hspace{-11mm}
\begin{picture}(450,72)
\glay
\multiput(139.4,7)(0,0.48){96}{\line(1,0){57.3}}
\color{white}
\multiput(139.3,41.2)(0,0.44){3}{\line(2,1){25.6}}
\multiput(165.5,53.2)(0.22,0.22){3}{\line(1,-1){31}}
\multiput(172,7)(0.17,-0.34){3}{\line(2,1){24.5}}
\multiput(139.6,37.8)(-0.22,-0.22){3}{\line(1,-1){31}}
\multiput(140,47.8)(4,2){5}{\circle*{11}}
\multiput(172,53)(3.8,-3.8){8}{\circle*{10}}
\multiput(186.3,53)(4,-4){4}{\circle*{15}}
\multiput(177,2.8)(3,1.5){8}{\circle*{12}}
\multiput(138,32)(3.8,-3.8){8}{\circle*{10}}
\multiput(142,18)(4,-4){3}{\circle*{14}}
\put(142,10){\circle*{10}}
\color{black}
\linethickness{0.1mm}
\multiput(138,0)(0,10){7}{\line(1,0){60}}
\multiput(138,0)(10,0){7}{\line(0,1){60}}
\linethickness{0.4pt}
\multiput(138.2,39.9)(-0.2,0.22){2}{\line(2,1){40}}
\multiput(157.9,59.9)(0.22,0.22){2}{\line(1,-1){40}}
\multiput(158.2,-0.1)(-0.2,0.22){2}{\line(2,1){40}}
\multiput(137.8,40)(0.22,0.22){2}{\line(1,-1){40}}
\put(184.1,57.5){\footnotesize $\times$}
\put(194.1,57.5){\footnotesize $\times$}
\put(194.1,47.5){\footnotesize $\times$}
\put(194.1,7.5){\footnotesize $\times$}
\put(194.1,-2.5){\footnotesize $\times$}
\put(184.1,-2.5){\footnotesize $\times$}
\put(144.1,-2.5){\footnotesize $\times$}
\put(134.1,-2.5){\footnotesize $\times$}
\put(134.1,7.5){\footnotesize $\times$}
\put(134.1,47.5){\footnotesize $\times$}
\put(134.1,57.5){\footnotesize $\times$}
\put(144.1,57.5){\footnotesize $\times$}
\put(138,40){\circle*{3}}
\put(198,20){\circle*{3}}
\put(201.5,59){$O$}
\put(125.5,37){$P$}
%
\glay
\multiput(249.3,5.6)(0,0.48){92}{\line(1,0){57.3}}
\color{white}
\multiput(249.3,21.6)(-0.22,0.22){3}{\line(1,1){28}}
\multiput(278.2,49.9)(0.22,0.22){3}{\line(1,-1){28}}
\multiput(279.8,5.5)(0.17,-0.34){3}{\line(2,1){26.7}}
\multiput(249.3,18.7)(-0.17,-0.34){3}{\line(2,-1){26.7}}
\multiput(249,27.8)(3.8,3.8){7}{\circle*{10}}
\multiput(249,40)(5,5){3}{\circle*{15}}
\multiput(283,52)(3.8,-3.8){8}{\circle*{10}}
\multiput(296.2,51)(5,-5){3}{\circle*{14}}
\multiput(285,3)(4.1,2){6}{\circle*{10}}
\put(304,8){\circle*{10}}
\multiput(249,14)(4,-2){6}{\circle*{10}}
\put(253,6){\circle*{10}}
\color{black}
\linethickness{0.1mm}
\multiput(248,0)(0,10){7}{\line(1,0){60}}
\multiput(248,0)(10,0){7}{\line(0,1){60}}
\linethickness{0.4pt}
\multiput(248.1,20)(-0.22,0.22){2}{\line(1,1){39.9}}
\multiput(267.9,59.9)(0.22,0.22){2}{\line(1,-1){40}}
\multiput(247.7,19.9)(0.1,0.2){2}{\line(2,-1){40}}
\multiput(268.2,-0.1)(-0.1,0.2){2}{\line(2,1){40}}
\put(294.1,57.5){\footnotesize $\times$}
\put(304.1,57.5){\footnotesize $\times$}
\put(304.1,47.5){\footnotesize $\times$}
\put(304.1,7.5){\footnotesize $\times$}
\put(304.1,-2.5){\footnotesize $\times$}
\put(294.1,-2.5){\footnotesize $\times$}
\put(254.1,-2.5){\footnotesize $\times$}
\put(244.1,-2.5){\footnotesize $\times$}
\put(244.1,7.5){\footnotesize $\times$}
\put(244.1,47.5){\footnotesize $\times$}
\put(244.1,57.5){\footnotesize $\times$}
\put(254.1,57.5){\footnotesize $\times$}
\put(248,20){\circle*{3}}
\put(308,20){\circle*{3}}
\put(311.5,59){$O$}
\put(235.5,17){$P$}
%
\end{picture}
\caption{}\label{fig31}
\end{figure}
%
If $q'=6$ and $P=(-6,-2)$, then either $(-3,-1)$ or $(-4,-1)$ 
is contained in Int\,$\square_{C}$. Besides, either $(-2,-5)$ or 
$(-3,-5)$ is contained in Int\,$\square_{C}$. On the other hand, 
if $q'=6$ and $P=(-6,-4)$, then either $(-2,-5)$ or $(-4,-5)$ 
is contained in Int\,$\square_{C}$. Hence, in each case, 
we obtain $g\geq 11$. 

By the above consideration, we can assume that $O$ is contained in 
$\square_{C}$. There exists an integer $a$ with $-6\leq a\leq 0$ such 
that $l(C,(1,-1))$ passes through $(0,a)$. We first 
remark that the cases $a=-1,-5$ do not occur by the relative minimality. 
If $a=0$, then by the assumption $g=10$, $\square_{C}$ must be a triangle 
with vertices $O$, $(-6,-6)$ and $(-6,0)$. This contradicts the 
assumption that $C$ is not a plane curve. We obtain the same 
contradiction if $a=-6$. Let us consider the case $a=-2$. Then either 
$(-q'+2,0)$ or $(-q',-2)$ is contained in $\square_{C}$. 
We define $B$ as a domain surrounded by the five 
segments $l((0,0),(-4,-6))$, $l((0,-2),(-q',-6))$, 
$l((-4,-6),(-q',0))$, $l((-q',-6),(-q'+2,0))$ and $l((-q',-2),(0,0))$. 
Since $\square_{C}$ includes $B$, we obtain $g\geq 11$ if $q'\geq 7$ 
(see Fig. \ref{fig32}). 
%
\begin{figure}[h]\hspace{-12mm}
\begin{picture}(450,80)
\glay
\multiput(102.3,31)(0,0.48){95}{\line(1,0){57}}
\color{white}
\multiput(159,76.7)(0,0.42){3}{\line(-4,-1){20}}
\multiput(139.5,71.1)(0,0.42){3}{\line(-4,-1){20}}
\multiput(120,65.5)(0,0.42){4}{\line(-4,-1){18}}
\multiput(129.7,30.7)(0.42,0){3}{\line(2,3){30}}
\multiput(114,30.8)(0,-0.45){3}{\line(2,1){12}}
\multiput(126,37.7)(0,-0.45){3}{\line(2,1){12}}
\multiput(102.2,51.5)(-0.45,0){3}{\line(1,-2){11}}
\multiput(101,51.5)(-0.3,0){7}{\line(1,3){4}}
\multiput(102,67)(3.6,0.9){14}{\circle*{10}}
\multiput(102,74)(3.6,0.9){5}{\circle*{10}}
\multiput(137,30)(2,3){14}{\circle*{12}}
\multiput(151,30)(2,3){7}{\circle*{14}}
\multiput(121,29)(4,2.1){5}{\circle*{10}}
\multiput(100,45)(2,-4){5}{\circle*{10}}
\color{black}
\linethickness{0.1mm}
\multiput(90,17)(0,10){7}{\line(1,0){70}}
\multiput(90,17)(10,0){8}{\line(0,1){60}}
\linethickness{0.4pt}
\multiput(114.4,11.1)(-0.22,0.22){3}{\line(1,1){56}}
\multiput(89.8,77)(0.25,0.15){2}{\line(1,-2){30}}
\multiput(120.2,17)(-0.25,0.15){2}{\line(2,3){40}}
\multiput(90,17)(-0.25,0.15){2}{\line(1,3){20}}
\put(160,77){\circle*{3}}
\put(173,69){$l(C,(1,-1))$}
\glay
\put(120.1,57){\line(1,0){9}}
\put(120,47.7){\line(0,1){6}}
\color{black}
\put(163,77){$O$}
\put(119,49){$B$}
\put(104,0){$q'\geq 7$}
\linethickness{0.6pt}
\qbezier[600](90,17)(125,37)(160,57)
\qbezier[600](90,57.05)(125,67.05)(160,77.05)
\glay
\multiput(280.4,28.5)(0,0.48){101}{\line(1,0){48.5}}
\color{white}
\multiput(280.4,61)(0,0.42){3}{\line(3,1){48}}
\multiput(298,28.2)(0.42,0){3}{\line(2,3){31}}
\multiput(287.4,28)(0,-0.42){4}{\line(3,2){17}}
\multiput(280.3,45)(-0.42,0){3}{\line(1,-3){6}}
\multiput(280.5,50)(-0.2,0.07){15}{\line(1,3){4}}
\multiput(280.4,67)(3,1){15}{\circle*{10}}
\multiput(280.4,74)(3,1){4}{\circle*{10}}
\multiput(305,27)(2.1,3.2){14}{\circle*{13}}
\multiput(320,27)(2.1,3.2){7}{\circle*{15}}
\multiput(295,27)(3,2){4}{\circle*{10}}
\multiput(277.8,37.6)(1,-3){4}{\circle*{10}}
\color{black}
\linethickness{0.1mm}
\multiput(270,17)(0,10){7}{\line(1,0){60}}
\multiput(270,17)(10,0){7}{\line(0,1){60}}
\linethickness{0.4pt}
\multiput(284.4,11.1)(-0.22,0.22){3}{\line(1,1){56}}
\multiput(269.9,77)(0.22,0.15){2}{\line(1,-3){20}}
\multiput(290.2,17)(-0.25,0.15){2}{\line(2,3){40}}
\multiput(270.2,17)(-0.15,0.24){2}{\line(3,2){60}}
\multiput(270.1,17)(-0.22,0.15){2}{\line(1,3){20}}
\multiput(270,56.8)(-0.15,0.22){2}{\line(3,1){60}}
\put(330,77){\circle*{3}}
\put(343,69){$l(C,(1,-1))$}
\glay
\put(296,57){\line(1,0){9}}
\put(300,47.7){\line(0,1){11}}
\color{black}
\put(333,77){$O$}
\put(295,49){$B$}
\put(287,0){$q'=6$}
\end{picture}
\caption{}\label{fig32}
\end{figure}
%
In the case $q'=6$, we can observe that 
there are at least eight lattice points in Int\,$\square_{C}$. 
Note that either $(-2,-3)$ or $(-3,-4)$ is contained in 
Int\,$\square_{C}$. Moreover, if 
$(-4,0)\in \square_{C}$ (resp. $(-6,-2)\in \square_{C}$), 
then $(-3,-1)$ and $(-4,-1)$ (resp. $(-5,-2)$ and $(-5,-3)$) are 
contained in Int\,$\square_{C}$. Hence, we have $g\geq 11$ in this 
case also. 

We next consider the case $a=-4$. We can take a lattice point 
$Q\in l(C,(-1,0))\cap \square_{C}$. In the cases 
$Q=(-q',0),(-q',-1),(-q',-2),(-q',-3),(-q',-4)$, we 
define $B_{1}$ as a domain surrounded by the four 
segments $l(O,(-2,$ $-6)),l((0,-4),(-q',-6))$, $l((-2,-6),Q)$ 
and $l(Q,O)$. Then we see that there exist more 
than ten lattice points in the interior of $B_{1}$ (that is, $g\geq 11$) 
except for the case where $q'=6$ and $Q=(-6,-4)$ (see Fig. \ref{fig34}). 
%
\begin{figure}[h]\hspace{-11mm}
\begin{picture}(450,67)
\glay
\multiput(194.4,13.6)(0,0.48){105}{\line(1,0){57.6}}
\color{white}
\multiput(194.3,26.4)(-0.2,0.3){3}{\line(3,2){56.5}}
\multiput(236,13)(0.3,-0.1){3}{\line(1,3){16}}
\multiput(219,13.2)(0.1,-0.27){18}{\line(3,1){18.3}}
\multiput(194,24)(-0.1,-0.2){3}{\line(2,-1){22}}
\multiput(194,33)(3.3,2.2){16}{\circle*{12}}
\multiput(194,46)(3.3,2.2){10}{\circle*{12}}
\multiput(194.5,58)(3.3,2.2){4}{\circle*{12}}
\multiput(241.7,13)(1,3){15}{\circle*{10}}
\multiput(250,15)(1,3){4}{\circle*{10}}
\multiput(194,18)(4,-2){5}{\circle*{11}}
\color{black}
\linethickness{0.1mm}
\multiput(193,5)(0,10){7}{\line(1,0){60}}
\multiput(193,5)(10,0){7}{\line(0,1){60}}
\linethickness{0.4pt}
\multiput(193.2,25)(-0.2,0.2){2}{\line(3,2){60}}
\multiput(233,4.9)(-0.25,0.15){2}{\line(1,3){20}}
\multiput(193.2,4.9)(-0.15,0.25){2}{\line(3,1){59.9}}
\multiput(192.8,24.9)(0.2,0.22){2}{\line(2,-1){40}}
\multiput(228.4,0)(-0.22,0.22){3}{\line(1,1){36}}
\put(267,40){$l(C,(1,-1))$}
\put(193,25){\circle*{3}}
\put(253,65){\circle*{3}}
\put(256,63){$O$}
\put(181,22){$Q$}
\glay
\put(222.5,35){\line(1,0){9}}
\put(223,28.3){\line(0,1){11}}
\put(233,27){\line(0,1){7}}
\color{black}
\put(221,29.8){$B_{1}$}
\end{picture}
\caption{}\label{fig34}
\end{figure}
%
In this exceptional case, since either $(-3,-5)$ or $(-4,-5)$ is 
contained in Int\,$\square_{C}$, we obtain $g\geq 11$. 
When $Q$ is $(-q',-5)$ or $(-q',-6)$, we define $B_{2}$ as a 
domain surrounded by the five segments $l(O,(-2,-6))$, 
$l((0,-4),(-q',-6))$, $l((-2,-6),Q)$, $l(Q,(-q',-4))$ 
and $l((-q',-4),O)$. Then we obtain $g\geq 11$. 

Lastly, we consider the case $a=-3$. We define $E_{1}$ as a domain 
surrounded by the five segments $l((0,0),(-3,-6))$, $l((0,-3),(-q',-6))$, 
$l((-3,-6),(-q',0))$, $l((-q',-6),(-q'+3,0))$ and $l((-q',-3),(0,0))$. 
Since $\square_{C}$ includes $E_{1}$, we obtain $g\geq 11$ if 
$q'\geq 8$. In the case $q'=7$, there exist at least ten lattice 
points in the interior of $E_{1}$. Moreover, if $(-7,0)\in 
{\rm Int}\,\square_{C}$ (resp. $(-7,0)\notin {\rm Int}\,\square_{C}$), 
then we see that also $(-5,-2)$ 
(resp. $(-5,-3)$) is contained in Int\,$\square_{C}$. This means that 
$g\geq 11$. Let us consider the case $q'=6$. We denote 
by $R=(-6,b)$ the vertex of $\square_{C}$ on $l(C,(-1,0))$ whose 
$w$-coordinate is maximal. 
In the cases $b=-4,-5,-6$, a domain $E_{2}$ surrounded by the five 
segments $l(O,(-3,-6))$, $l((0,-3),(-6,-6))$, $l((-3,-6),R)$, 
$l(R,(-3,0))$ and $l((-6,-3),O)$ is included in $\square_{C}$ (see 
Fig. \ref{fig33}). 
%
%
\begin{figure}[h]\hspace{-11mm}
\begin{picture}(450,67)
\glay
\multiput(86.3,12)(0,0.48){106}{\line(1,0){58}}
\color{white}
\multiput(86.2,25.5)(-0.22,0.11){20}{\line(3,4){28}}
\multiput(98,40)(-0.15,0.3){3}{\line(2,1){45.5}}
\multiput(120,11.9)(0.3,0.15){3}{\line(1,2){24.5}}
\multiput(103,11.7)(0.15,-0.3){3}{\line(2,1){22}}
\multiput(86.3,22)(0,-0.43){3}{\line(3,-2){14.5}}
\multiput(86,39)(4,2){14}{\circle*{10}}
\multiput(86,48)(4,2){9}{\circle*{10}}
\multiput(85,57)(4,2){5}{\circle*{11}}
\multiput(125.5,12)(2,4){12}{\circle*{11}}
\multiput(137.5,14)(3,6){4}{\circle*{14}}
\multiput(108,9)(4,2){5}{\circle*{10}}
\multiput(86,16)(3.3,-2.2){4}{\circle*{11}}
\color{black}
\linethickness{0.1mm}
\multiput(85,3)(0,10){7}{\line(1,0){60}}
\multiput(85,3)(10,0){7}{\line(0,1){60}}
\linethickness{0.4pt}
\multiput(109.4,-2.9)(-0.22,0.22){3}{\line(1,1){44}}
\multiput(84.7,23)(0.15,0.24){2}{\line(3,-2){30}}
\multiput(115.2,3)(-0.24,0.12){2}{\line(1,2){30}}
\multiput(85.1,2.9)(-0.12,0.24){2}{\line(2,1){60}}
\multiput(85,22.8)(-0.3,0.17){2}{\line(3,4){30}}
\multiput(85.1,32.9)(-0.12,0.24){2}{\line(2,1){60}}
\put(145,63){\circle*{3}}
\put(85,23){\circle*{3}}
\glay
\put(108.7,33){\line(1,0){9}}
\put(115,27){\line(0,1){11}}
\color{black}
\put(148,62){$O$}
\put(108,28.5){$E_{2}$}
\put(73,20){$R$}
\glay
\multiput(196.1,9)(0,0.48){112}{\line(1,0){58}}
\color{white}
\multiput(196,15)(-0.3,0.2){3}{\line(3,5){28.5}}
\multiput(213.6,42.2)(-0.15,0.3){3}{\line(2,1){39.8}}
\multiput(228.5,8.8)(0.3,0.15){3}{\line(1,2){27}}
\multiput(208,9)(0.15,-0.3){3}{\line(2,1){27}}
\multiput(195,9)(0.07,0.2){16}{\line(3,-1){10.5}}
\multiput(191,16)(2.9,4.8){11}{\circle*{10}}
\multiput(196,37)(2.9,4.8){6}{\circle*{10}}
\multiput(195,50)(2.9,4.8){4}{\circle*{12}}
\multiput(213.3,48)(4,2){10}{\circle*{10}}
\put(226,60){\circle*{12}}
\multiput(234,10)(2,4){12}{\circle*{10}}
\multiput(243,10.2)(2,4){7}{\circle*{10}}
\multiput(252,10)(2,4){3}{\circle*{10}}
\multiput(213,5.8)(4,2){6}{\circle*{10}}
\color{black}
\linethickness{0.1mm}
\multiput(195,3)(0,10){7}{\line(1,0){60}}
\multiput(195,3)(10,0){7}{\line(0,1){60}}
\linethickness{0.4pt}
\multiput(219.4,-2.9)(-0.22,0.22){3}{\line(1,1){44}}
\multiput(194.7,12.9)(0.15,0.22){2}{\line(3,-1){30}}
\multiput(225.2,3)(-0.24,0.12){2}{\line(1,2){30}}
\multiput(195.1,2.9)(-0.12,0.24){2}{\line(2,1){60}}
\multiput(195,12.7)(-0.3,0.1){2}{\line(3,5){30}}
\multiput(195.1,32.9)(-0.12,0.24){2}{\line(2,1){60}}
\put(255,63){\circle*{3}}
\put(195,13){\circle*{3}}
\glay
\put(218.5,33){\line(1,0){9.5}}
\put(225,27){\line(0,1){11}}
\color{black}
\put(258,62){$O$}
\put(218,28.5){$E_{2}$}
\put(183,10){$R$}
\glay
\multiput(306,4)(0,0.48){122}{\line(1,0){58}}
\color{white}
\multiput(305.8,5)(-0.32,0.16){3}{\line(1,2){28.5}}
\multiput(326,44)(-0.16,0.32){3}{\line(2,1){37}}
\multiput(336,3.8)(0.32,0.16){4}{\line(1,2){29}}
\multiput(307.4,3.8)(0.16,-0.32){3}{\line(2,1){36.5}}
\multiput(302.5,10)(2,4){14}{\circle*{10}}
\multiput(305,31)(2,4){9}{\circle*{10}}
\multiput(305,48)(2,4){4}{\circle*{12}}
\multiput(326.4,49.2)(4,2){9}{\circle*{10}}
\put(337,58){\circle*{14}}
\multiput(343,3.8)(2,4){13}{\circle*{13}}
\multiput(355,6)(2,4){6}{\circle*{13}}
\put(362,6){\circle*{10}}
\multiput(312,1)(4,2){9}{\circle*{10}}
\put(332,5){\circle*{12}}
\color{black}
\linethickness{0.1mm}
\multiput(305,3)(0,10){7}{\line(1,0){60}}
\multiput(305,3)(10,0){7}{\line(0,1){60}}
\linethickness{0.4pt}
\multiput(329.4,-2.9)(-0.22,0.22){3}{\line(1,1){44}}
\multiput(335.2,3)(-0.24,0.12){2}{\line(1,2){30}}
\multiput(305.1,2.9)(-0.12,0.24){2}{\line(2,1){60}}
\multiput(305,2.9)(-0.24,0.12){2}{\line(1,2){30}}
\multiput(305.1,32.9)(-0.12,0.24){2}{\line(2,1){60}}
\put(365,63){\circle*{3}}
\put(305,3){\circle*{3}}
\glay
\put(329,33){\line(1,0){10}}
\put(335,27){\line(0,1){11}}
\color{black}
\put(368,62){$O$}
\put(328,28.5){$E_{2}$}
\put(293,0){$R$}
\end{picture}
\caption{}\label{fig33}
\end{figure}
%
Note that, in each case, either $(-1,-2)$ or $(-4,-5)$ is contained in 
Int\,$\square_{C}$. Moreover, in the case $b=-6$, 
either $(-2,-1)$ or $(-5,-4)$ is contained in Int\,$\square_{C}$. 
We thus obtain $g\geq 11$. If $-3\leq b\leq 0$, we define $E_{3}$ as 
a domain surrounded by the four segments $l(O,(-3,-6))$, 
$l((0,-3),(-6,-6))$, $l((-3,-6),R)$ and $l(R,O)$. If $-2\leq b\leq 0$, 
then there exist more than eleven lattice points in the interior 
of $E_{3}$. Assume that $b=-3$. If $(-3,-5)$ is not contained 
in ${\rm Int}\,\square_{C}$, then by a simple consideration, we see 
that $(-1,-2)$ and $(-2,-4)$ are contained in Int\,$\square_{C}$. 
If $(-3,-5)$ is contained in ${\rm Int}\,\square_{C}$, it is clear 
that the equality $g=10$ holds if and only if $R=(-6,-3)$ and 
$\square_{C}$ is a triangle with vertices $O$, $(-3,-6)$ and $(-6,-3)$. 
\end{proof}
Similarly to Lemma \ref{g10}, we obtain the following lemma. 
\begin{lemma}\label{g4}
Let $C$ be a curve as in Theorem \ref{mthm}, and assume that $(S,C)$ is 
relatively minimal. If $q=4$ and $g=4$, then $\square_{C}$ is the first 
triangle in Fig. \ref{fig35}. 
\end{lemma}
We are now ready to prove the main theorem. 
\begin{proof}[\indent\sc Proof of Theorem \ref{mthm}.]
As mentioned at the beginning of the proof of Proposition \ref{kandata}, 
the statements (i) and (ii) are obvious if $q=2$. Hence we assume 
that $q\geq 3$. \\
(i) Let $\varphi$ be a morphism from $S$ to $\mathbb P^{1}$ 
of ${\rm deg}\hspace{0.3mm}\varphi|_{C}=k$ whose existence is guaranteed 
by Proposition \ref{kandata}. We shall show that $\varphi$ is 
a toric fibration of $S$. We denote by $F$ the fiber of $\varphi$. 
Since $C.(F-C)\leq k-\frac{3}{4}k^{2}<0$, we have $h^{0}(C,F-C)=0$. 
Hence the cohomology exact sequence 
$$0\rightarrow H^{0}(S,F-C)\rightarrow H^{0}(S,F)\rightarrow 
H^{0}(C,F|_{C})\rightarrow \cdots $$
implies that $h^{0}(C,F|_{C})\geq h^{0}(S,F)$. Hence we have 
$h^{0}(S,F)\leq 2$. Namely, $\square_{F}$ is a segment connecting two 
lattice points. We denote these points by $O=(0,0)$ and $P=(z,w)$, 
where $z$ and $w$ are integers such that $(|z|,|w|)=1$. Then the 
point $(-w,z)$ must be contained in ${\rm Pr}^{\ast}(S)$. Therefore, by 
Fact \ref{jinmencho}, we see that $F$ is a fiber of some toric fibration. 

In what follows, by reembedding if necessary, we may assume 
that $(S,C)$ is relatively minimal. \\
(ii) If $(g,q)\neq (4,4),(5,4),(10,6)$, we have $k=q$ by 
(i). Assume that $g=q=4$. In this case, we have $k\leq 
\left \lfloor\frac{g+3}{2}\right \rfloor=3$ and $C^{2}\geq 12$ by 
Proposition \ref{k4}. 
Suppose that $k=2$. Then by Theorem \ref{serrano}, there exists an 
effective divisor $V$ on $S$ satisfying the properties ({\tt a}) and 
({\tt b}) for $X=C$ and $p=2$. Recall that, as we saw at the beginning of 
the proof of Proposition \ref{kandata}, $V^{2}$ must be more than two. 
Since this contradicts the property ({\tt a}), we can conclude that 
$k=3$. In the case where $g=5$ and $q=4$, similarly to the previous 
case, we can show that the supposition $k\leq 3$ yields a contradiction. 
If $g=10$ and $q=6$, then we 
obtain $k=6$ by (iv) which is proved below. 
\\
(iii) 
By \cite{acgh}, ${\rm Cliffdim}(C)=2$ if and only if $C$ is 
isomorphic to a plane curve of degree $\geq 5$. Besides, 
if ${\rm Cliffdim}(C)\geq 3$, then $g\geq 10$ holds. Hence, we 
have ${\rm Cliffdim}(C)=1$ in the cases $(g,q)=(4,4),(5,4)$. 
In other cases (except for the case $(g,q)=(10,6)$), we see that the 
number of gonality pencils on $C$ is finite by (i). It follows that 
${\rm Cliffdim}(C)=1$. \\
(iv) By Lemma \ref{g10} and Fig. \ref{fig29}, in this case, we can 
see the explicit structures of $S$ and $C$. First, 
$|-K_{S}|$ has no base points, $h^{0}(S,-K_{S})=4$ and $(-K_{S})^{2}=3$. 
Besides, we can write $C\sim -3K_{S}$, that is, $C.(-K_{S})=9$. We 
consider the morphism $\Phi _{|-K_{S}|}:S\rightarrow {\mathbb P}^{3}$, 
and put $T=\Phi_{|-K_{S}|}(S)$. Then, by the equality 
$${\rm deg}\hspace{0.3mm}\Phi _{|-K_{S}|}\cdot 
{\rm deg}\hspace{0.3mm}T=(-K_{S})^{2}=3,$$
we obtain ${\rm deg}\hspace{0.3mm}\Phi _{|-K_{S}|}=1$ and 
${\rm deg}\hspace{0.3mm}T=3$. We denote by $H$ a 
hyper plane of $T$. The short exact 
sequence of sheaves $0\rightarrow {\cal O}_{\mathbb P^{3}}\rightarrow 
{\cal O}_{\mathbb P^{3}}(3H)\rightarrow {\cal O}_{T}(3H)\rightarrow 0$ 
induces the surjection 
$H^{0}({\mathbb P^{3}},3H)\rightarrow H^{0}(T,3H|_{T})=H^{0}(T,C)$, 
where we abuse notation to denote the image of $C$ 
under $\Phi _{|-K_{S}|}$ by the same symbol. 
Hence we see that $C$ is an irreducible component of $T\cap T'$, 
where $T'$ is a cubic surface in $\mathbb P^{3}$. Since 
$$9={\rm deg}\hspace{0.3mm}T\cap T'\geq 
{\rm deg}\hspace{0.3mm}C=C.(-K_{S})=9,$$
we can conclude that $C=T\cap T'$. 
\end{proof}
%
%
%
\section{The case where {\boldmath $(g,q)=(4,4)$}}\label{exam}

In this section, we focus on the exceptional curve in 
Theorem \ref{mthm} (ii), and exhibit its structure. Let $S$, 
$C$ and $q$ be as in Theorem \ref{mthm}, and assume $g=q=4$. By 
Lemma \ref{g4}, the fan $\Delta_{S}$ and the lattice 
polygon $\square_{C}$ are as in Fig. \ref{fig13}. 
%
\begin{figure}[h]\hspace{-10mm}
\begin{picture}(450,82)
\setlength\unitlength{1.2pt}
\linethickness{0.1mm}
\multiput(115,19)(0,10){4}{\line(1,0){50}}
\multiput(125,11)(10,0){4}{\line(0,1){48}}
\color{white}
\put(125,57.5){\line(0,1){4}}
\color{black}
\multiput(115,38.4)(0,0.22){6}{\line(1,0){50}}
\multiput(144.4,11)(0.22,0){6}{\line(0,1){48}}
\multiput(127,56.2)(0.27,0.27){4}{\line(1,-1){36}}
\multiput(145.4,38.7)(-0.27,0.27){4}{\line(1,1){17}}
\multiput(144.5,39)(0.36,0.18){3}{\line(1,-2){13}}
\multiput(144.5,38.7)(0.18,0.36){3}{\line(-2,1){27}}
\put(133,62.5){\small $\sigma(D_{1})$}
\put(158,59){\small $\sigma(D_{2})$}
\put(109.5,60){\small $\sigma(D_{9})$}
\put(145,49){\circle*{3}}
\put(155,49){\circle*{3}}
\put(155,39){\circle*{3}}
\put(155,29){\circle*{3}}
\put(155,19){\circle*{3}}
\put(145,29){\circle*{3}}
\put(135,39){\circle*{3}}
\put(125,49){\circle*{3}}
\put(135,49){\circle*{3}}
\put(174.3,51){\circle*{2}}
\put(175,45.5){\circle*{2}}
\put(175.2,40){\circle*{2}}
\put(107,52){\circle*{2}}
\put(105.6,46.5){\circle*{2}}
\put(105,41){\circle*{2}}
\put(140,0){$\Delta_{S}$}
\multiput(210,13)(0,10){5}{\line(1,0){40}}
\multiput(210,13)(10,0){5}{\line(0,1){40}}
\multiput(210.4,12.7)(-0.15,0.3){4}{\line(2,1){40}}
\multiput(229.4,52.7)(0.27,0.27){4}{\line(1,-1){20}}
\multiput(210.3,12.8)(-0.3,0.15){4}{\line(1,2){20}}
\color{white}
\put(224.5,33){\line(1,0){14}}
\put(230,30){\line(0,1){9}}
\color{black}
\put(225,31){$\square_{C}$}
\end{picture}
\caption{}\label{fig13}
\end{figure}
%
Considering the shape of $\square_{C}$, we can take a plane model 
$$C':x^{4}y^{2}+
x^{2}y^{4}+ax^{2}y^{2}z^{2}=z^{6}\ \ (a\in {\mathbb C})$$
of $C$. We shall denote the pull-backs on $S$ of functions $x$, $y$ 
and $z$ by same symbols. Since $K_{S}\sim -\Sigma_{i=1}^{9}D_{i}$ 
and $C\sim -2K_{S}$, we obtain $h^{0}(S,K_{S})=h^{1}(S,K_{S})=0$, 
which implies that $H^{0}(C,K_{C})
\simeq H^{0}(S,-K_{S})=\langle x^{2}y, xy^{2}, xyz, z^{3}\rangle $. 
Hence the restriction of the rational map 
\begin{eqnarray*}
\begin{array}{rll}
{\mathbb P^{2}} &\!\!\! -\hspace{-1.9mm}\rightarrow 
\hspace{-6.55mm}\raisebox{0.3mm}
{\scriptsize \color{white} $\bullet \hspace{0.75mm}\bullet$}
 &\!\!\! {\mathbb P^{3}}\\
\ \ \ (x:y:z) &\!\!\! \longmapsto &\!\!\! (x^{2}y:xy^{2}:xyz:z^{3})
\end{array}
\end{eqnarray*}
to $C'$ gives the canonical embedding of $C$. 
Let $(t:u:v:w)$ be a homogeneous coordinate system in $\mathbb P^{3}$. 
Then the canonical curve of $C$ lies on a quadric surface 
$T:t^{2}+u^{2}+av^{2}=w^{2}$. Thus if $a\neq 0$, then 
two families of lines on $T$ cut out two distinct pencils $g_{3}^{1}$ and 
$h_{3}^{1}$ on $C$. 
On the other hand, if $a=0$, then $T$ is a quadric cone, and one family 
of lines cuts out a unique $g_{3}^{1}$ on $C$. In sum, we can conclude 
that 
\begin{itemize}
\setlength{\itemsep}{-2pt}
\item[(i)]
If $a\neq 0$, $C$ is a curve of bidegree $(3,3)$ on 
$\mathbb P^{1}\times \mathbb P^{1}$. 
\item[(ii)]
If $a=0$, $C$ is linearly equivalent to 
$3\Delta _{0}+6F$ on $\Sigma_{2}$, where $\Delta _{0}$ and $F$ denote 
the minimal section and the fiber of the ruling of $\Sigma_{2}$, 
respectively. 
\end{itemize}
Unfortunately, however, we can not distinguish the above difference 
from the information of the lattice polygon. 
%
%
%
\section{Application}\label{appl}

By combining Theorem \ref{mthm} with results in \cite{kaw}, we can 
compute Weierstrass gap sequences at ramification points (with high 
ramification indexes) of a gonality pencil. 
For example, in this section, we consider trigonal curves and 
provide a geometric interpretation of the structure of gap sequences 
at ramification points. 
Let us review the preliminary results. 
Firstly, it is known that a trigonal covering of $\mathbb P^{1}$ has 
four types of gap sequences. 
\begin{theorem}[\cite{cop1,cop2}]\label{gap}
Let $C$ be a smooth trigonal curve of genus $g$ and Maroni 
invariant $m$, 
and $P$ a ramification point of a trigonal covering from $C$ to 
${\mathbb P}^{1}$. Then the Weierstrass gap sequence at $P$ is one of 
the following types. \\[2mm]
\ \ In the case where $P$ is a total ramification point: 
\vspace{-2mm}
\begin{itemize}
\setlength{\itemsep}{-2pt}
\setlength{\itemindent}{30pt}
\item[{\rm type\ I}\,] 
$\{1,2,4,5,\ldots ,3m+1,3m+2,3m+4,3m+7,\ldots ,3(g-m)-5\}$, 
\item[{\rm type\ I\hspace{-0.3mm}I}] 
$\{1,2,4,5,\ldots ,3m+1,3m+2,3m+5,3m+8,\ldots ,3(g-m)-4\}$. 
\end{itemize}
\ \ In the case where $P$ is an ordinary ramification point: 
\vspace{-2mm}
\begin{itemize}
\setlength{\itemsep}{-2pt}
\setlength{\itemindent}{30pt}
\item[{\rm type\ I}\,] 
$\{1,2,3,\ldots ,2m+1,2m+2,2m+3,2m+5,\ldots ,2(g-m)-3\}$, 
\item[{\rm type\ I\hspace{-0.3mm}I}] 
$\{1,2,3,\ldots ,2m+1,2m+2,2m+4,2m+6,\ldots ,2(g-m)-2\}$. 
\end{itemize}
\end{theorem}
%
%
Besides, Kato and Horiuchi presented the following criterion for 
distinguishing the above types. 
\begin{theorem}[\cite{kat}]\label{katoh}
Let $C$ be a trigonal curve of genus $g\geq 5$ and Maroni 
invariant $m$. Then $C$ has a plane model defined by 
\begin{eqnarray}\label{genjutsushi}
y^{3}+x^{\mu}A(x)y+x^{\nu}B(x)=0,
\end{eqnarray}
where ${\rm deg}A(x)+\mu =2m+4$, ${\rm deg}B(x)+\nu =3m+6$ and 
$A(0)B(0)\neq 0$. 
\begin{itemize}
\setlength{\itemsep}{-2pt}
\item[{\rm (i)}] If $\mu \geq \nu =1$, there exists a total ramification 
point of type {\rm I} over $x=0$. 
\item[{\rm (ii)}] If $\mu \geq \nu =2$, there exists a total ramification 
point of type {\rm I\hspace{-0.3mm}I} over $x=0$. 
\item[{\rm (iii)}] If $\mu =\nu =0$ and the order of zero 
of $4A(x)^{3}+27B(x)^{2}$ at $x=0$ is odd, there exists an ordinary 
ramification point of type {\rm I} over $x=0$. 
\item[{\rm (iv)}] If $\nu >\mu =1$, there exists an ordinary 
ramification point of type {\rm I\hspace{-0.3mm}I} over $x=0$. 
\item[{\rm (v)}] Otherwise, there exist no ramification points 
over $x=0$. 
\end{itemize}
\end{theorem}
%
%
In order to interpret Theorem \ref{katoh} in terms of lattice polygons, 
%
we transform the defining equation 
(\ref{genjutsushi}) in the case (iii). Note that we need the 
condition $m<(g-2)/2$ in this case, since $2(g-m)-3$ must be at least $g$. 
We can write two polynomials as 
$A(x)=\sum_{i=1}^{2m+4}a_{i}x^{i}-3k^{2}$ and 
$B(x)=\sum_{i=1}^{3m+6}b_{i}x^{i}+2k^{3}$, where $k\neq 0$. We formally 
set $a_{2m+5}=\cdots =a_{3m+6}=0$, and define 
$\alpha ={\rm min}\{i\mid a_{i}\neq 0\}$, 
$\beta ={\rm min}\{i\mid ka_{i}+b_{i}\neq 0\}$, 
%
%
$A_{1}(x)=\sum_{i=\beta}^{3m+6}a_{i}x^{i}$, 
$B_{1}(x)=\sum_{i=\beta}^{3m+6}b_{i}x^{i}$ and 
$E(x)=A(x)-A_{1}(x)+3k^{2}$. 
%
Then we have ${\rm mindeg}(kA_{1}(x)+B_{1}(x))=\beta $, where the 
notation `{\rm mindeg}' denotes the minimal degree of a polynomial. 
Since $m>g/2$, we see that $\beta $ is less than $2m+4$. 
Let us check that ${\rm mindeg}(4A(x)^{3}+27B(x)^{2})=\beta <2\alpha$. 
By a simple computation, we have 
\begin{eqnarray}\label{ruzami}
\begin{array}{cl}
 &\!\! 4A(x)^{3}+27B(x)^{2}\\
= &\!\! 4(E(x)+A_{1}(x))^{3}+27(-kE(x)+B_{1}(x))^{2}-
36k^{2}(E(x)+A_{1}(x))^{2}\\
 &\!\! +108k^{3}(kA_{1}(x)+B_{1}(x)). 
\end{array}
\end{eqnarray}
%
%
%
In the case where $E(x)=0$, $\beta$ is equal to $\alpha$ by 
definition, and the equality ${\rm mindeg}(4A(x)^{3}+27B(x)^{2})=\beta$ 
follows from (\ref{ruzami}). On the other hand, if $E(x)\neq 0$, 
we have ${\rm mindeg}(4A(x)^{3}+27B(x)^{2})=
{\rm min}\{2{\rm mindeg}E(x),\beta \}=\beta $ by (\ref{ruzami}) 
and its oddness. Note that ${\rm mindeg}E(x)=\alpha$ in 
this case. Consequently, if we perform a coordinate 
transformation $y'=y-k$, and put $y'=y$ again, then the 
defining equation (\ref{genjutsushi}) is rewritten as 
\begin{eqnarray}\label{bubbleslime}
y^{3}+3ky^{2}+x^{\alpha}C(x)y+x^{\beta}D(x)=0,
\end{eqnarray}
where $\alpha$ and $\beta$ are positive integers such that 
$\beta$ is odd and $\beta <{\rm min}\{2\alpha,2m+4\}$, 
${\rm deg}C(x)+\alpha =2m+4$, ${\rm deg}D(x)+\beta =3m+6$ and 
$C(0)D(0)\neq 0$. 
%
%
%
%
%

Now we are in a position to translate Theorem \ref{katoh} in terms of the 
geometry of lattice polygons. We embed $C$ in a toric surface by blowing 
up repeatedly. 
Then the lattice polygon $\square_{C}$ associated to $C$ is drawn as 
in Fig. \ref{fig37}. 
%
\begin{figure}[h]\hspace{-5mm}
\begin{picture}(450,136)
\setlength\unitlength{1.2pt}
\linethickness{0.1mm}
\multiput(35,78)(0,10){4}{\line(1,0){120}}
\multiput(35,78)(10,0){13}{\line(0,1){30}}
\color{white}
\put(58,88){\line(1,0){22}}
\put(65,82){\line(0,1){12}}
\put(75,82){\line(0,1){12}}
\put(108,98){\line(1,0){42}}
\multiput(115,91)(10,0){4}{\line(0,1){10}}
\color{black}
\multiput(34.6,108)(0.33,0.11){4}{\line(1,-3){10}}
\multiput(34.6,107.7)(0.1,0.3){4}{\line(4,-1){120}}
\multiput(44.6,77.4)(0,0.29){5}{\line(1,0){110}}
\put(45,78){\circle*{3}}
\put(155,78){\circle*{3}}
\put(55,88){\circle*{3}}
\put(115,88){\circle*{3}}
\put(35,108){\circle*{3}}
\put(35,78){\circle*{3}}
\put(13,106){$(0,3)$}
\put(59,86){$(\mu ,1)$}
\put(109,93.5){$(2m+4,1)$}
\put(38,68){$(1,0)$}
\put(135,68){$(3m+6,0)$}
\put(26,72){$O$}
\put(80,63){Case (i)}
\multiput(35,15)(0,10){4}{\line(1,0){120}}
\multiput(35,15)(10,0){13}{\line(0,1){30}}
\multiput(34.5,45)(0.29,0){5}{\line(0,-1){10}}
\multiput(34.6,34.6)(0.16,0.32){4}{\line(5,-2){50}}
\multiput(34.6,44.7)(0.1,0.3){4}{\line(4,-1){120}}
\multiput(84.6,14.4)(0,0.29){5}{\line(1,0){70}}
\color{white}
\put(61,25){\line(1,0){21.5}}
\linethickness{0.7mm}
\multiput(62,20.7)(1.7,0){11}{\line(0,1){10}}
\linethickness{0.4pt}
\put(108,35){\line(1,0){42}}
\multiput(115,28)(10,0){4}{\line(0,1){10}}
\color{black}
\put(85,15){\circle*{3}}
\put(155,15){\circle*{3}}
\put(85,25){\circle*{3}}
\put(115,25){\circle*{3}}
\put(35,45){\circle*{3}}
\put(35,35){\circle*{3}}
\put(35,15){\circle*{3}}
\put(13,45){$(0,3)$}
\put(13,31){$(0,2)$}
\put(62.3,23.5){$(\alpha ,1)$}
\put(109,30.5){$(2m+4,1)$}
\put(51,5){$(\beta ,0)$}
\put(72,7.7){\vector(2,1){9}}
\put(136,5){$(3m+6,0)$}
\put(26,9){$O$}
\put(80,0){Case (iii)}
\multiput(200,78)(0,10){4}{\line(1,0){120}}
\multiput(200,78)(10,0){13}{\line(0,1){30}}
\multiput(199.6,107.7)(0.32,0.22){4}{\line(2,-3){20}}
\multiput(199.6,107.7)(0.1,0.3){4}{\line(4,-1){120}}
\multiput(219.6,77.4)(0,0.29){5}{\line(1,0){100}}
\color{white}
\put(217.5,88){\line(1,0){21}}
\multiput(217.5,83.3)(0.4,0){12}{\line(0,1){9}}
\put(220,82.5){\line(0,1){12}}
\put(230,82.5){\line(0,1){12}}
\put(273,98){\line(1,0){42}}
\multiput(280,91)(10,0){4}{\line(0,1){10}}
\color{black}
\put(220,78){\circle*{3}}
\put(320,78){\circle*{3}}
\put(240,88){\circle*{3}}
\put(280,88){\circle*{3}}
\put(200,108){\circle*{3}}
\put(200,78){\circle*{3}}
\put(178,106){$(0,3)$}
\put(218,86){$(\mu ,1)$}
\put(274,93.5){$(2m+4,1)$}
\put(207,68){$(2,0)$}
\put(300,68){$(3m+6,0)$}
\put(191,72){$O$}
\put(245,63){Case (ii)}
\multiput(200,15)(0,10){4}{\line(1,0){120}}
\multiput(200,15)(10,0){13}{\line(0,1){30}}
\color{white}
\put(199,25){\line(1,0){8}}
\put(200,19.5){\line(0,1){12}}
\put(273,35){\line(1,0){42}}
\multiput(280,28)(10,0){4}{\line(0,1){10}}
\color{black}
\multiput(199.5,45)(0.34,0.17){4}{\line(1,-2){10}}
\multiput(209.5,24.7)(0.17,0.34){4}{\line(2,-1){20}}
\multiput(199.6,44.7)(0.1,0.3){4}{\line(4,-1){120}}
\multiput(229.6,14.4)(0,0.29){5}{\line(1,0){90}}
\put(230,15){\circle*{3}}
\put(320,15){\circle*{3}}
\put(210,25){\circle*{3}}
\put(280,25){\circle*{3}}
\put(200,45){\circle*{3}}
\put(200,15){\circle*{3}}
\put(178,43){$(0,3)$}
\put(187.5,23){$(1,1)$}
\put(274,30.5){$(2m+4,1)$}
\put(217,5){$(\nu ,0)$}
\put(300,5){$(3m+6,0)$}
\put(191,9){$O$}
\put(245,0){Case (iv)}
\end{picture}
\caption{}\label{fig37}
\end{figure}
%
In the case (i), the fan $\Delta_{S}$ associated to $S$ is as in 
Fig. \ref{fan2}. 
%
\begin{figure}[h]\hspace{-13mm}
\begin{picture}(400,109)
\setlength\unitlength{1.2pt}
\linethickness{0.1mm}
\multiput(160,22)(0,10){6}{\line(1,0){60}}
\multiput(170,12)(10,0){5}{\line(0,1){70}}
\multiput(160,31.4)(0,0.3){5}{\line(1,0){40}}
\multiput(199.4,12)(0.3,0){5}{\line(0,1){70}}
\multiput(199.4,32)(0.35,-0.1){4}{\line(1,4){12}}
\multiput(199.4,32.2)(0.22,-0.22){5}{\line(-1,-1){19}}
\multiput(199.4,32.2)(0.14,-0.28){5}{\line(-2,-1){34}}
\multiput(199.4,32.3)(0.1,-0.3){5}{\line(-3,-1){38}}
\put(200,42){\circle*{3}}
\put(210,72){\circle*{3}}
\put(200,22){\circle*{3}}
\put(190,22){\circle*{3}}
\put(180,22){\circle*{3}}
\put(170,22){\circle*{3}}
\put(190,32){\circle*{3}}
\put(182,87){$\sigma(D_{1})$}
\put(210,84.5){$\sigma(D_{2})$}
\put(136,16){$\sigma(D_{6})$}
\put(135,30){$\sigma(D_{7})$}
\color{white}
\put(213.5,72){\line(1,0){7}}
\color{black}
\put(214,69){$(1,m+2)$}
\put(160,10){\circle*{2}}
\put(165,7.5){\circle*{2}}
\put(170.5,6){\circle*{2}}
\put(187,0){$\Delta_{S}$}
%
\end{picture}
\caption{}\label{fan2}
\end{figure}
%
Considering the process of blowing-ups, we see that an intersection 
point $P=C\cap D_{6}$ is a unique point over 
the origin of the plane model (\ref{genjutsushi}). On the other hand, 
Theorem \ref{mthm} and Fact \ref{jinmencho} show that the fiber $F$ 
of a trigonal covering from $C$ to $\mathbb P^{1}$ is $F\sim D_{4}+
2D_{5}+3D_{6}+D_{7}$, which implies that $P$ is a total ramification 
point. By applying Corollary 1.6 in \cite{kaw}, we can to determine the 
gap sequence at $P$ as 
$$\{j\mid \mbox{the line $3X+Y=3+j$ has a lattice point in 
Int\,$\square_{C}$}\}.$$
For better understanding, we attach to each lattice point an 
integer $j$ such that the line $3X+Y=3+j$ passes through 
it (see Fig. \ref{lat}). 
%
\begin{figure}[h]\hspace{-14mm}
\begin{picture}(400,85)
\setlength\unitlength{1.2pt}
\multiput(50,10)(0,10){4}{\line(1,0){140}}
\multiput(60,0)(10,0){13}{\line(0,1){50}}
\multiput(59.7,40)(0.24,0.08){3}{\line(1,-3){10}}
\multiput(59.7,39.8)(0.08,0.22){3}{\line(4,-1){120}}
\multiput(70,9.6)(0,0.25){3}{\line(1,0){110}}
\multiput(77.5,56)(0.3,0.1){4}{\line(1,-3){19}}
\put(50,10){\vector(1,0){140}}
\put(60,0){\vector(0,1){50}}
\put(50,9.85){\line(1,0){140}}
\put(60.15,0){\line(0,1){50}}
\put(90,20){\circle*{3}}
\put(50,0.7){$O$}
\put(186,0.7){$X$}
\put(50,45){$Y$}
\put(35,60){$3X+Y=3+j$}
\multiput(210,10)(0,10){4}{\line(1,0){140}}
\multiput(220,0)(10,0){13}{\line(0,1){50}}
\multiput(219.6,40)(0.24,0.08){3}{\line(1,-3){10}}
\multiput(219.7,39.8)(0.08,0.22){3}{\line(4,-1){120}}
\multiput(230,9.6)(0,0.25){3}{\line(1,0){110}}
\put(210,10){\vector(1,0){140}}
\put(220,0){\vector(0,1){50}}
\put(210,9.85){\line(1,0){140}}
\put(220.15,0){\line(0,1){50}}
\put(210,0.7){$O$}
\put(346,0.7){$X$}
\put(210,45){$Y$}
\color{white}
\put(228,20){\line(1,0){4}}
\put(230,17){\line(0,1){5.5}}
\put(228,30){\line(1,0){4}}
\put(230,27){\line(0,1){5.5}}
\multiput(237.8,9.5)(0,0.3){3}{\line(1,0){4}}
\put(240,6.9){\line(0,1){5.5}}
\put(237.5,20){\line(1,0){4}}
\put(240,17){\line(0,1){5.5}}
\put(237.5,30){\line(1,0){4}}
\put(240,27){\line(0,1){5.5}}
\multiput(247.6,9.5)(0,0.3){3}{\line(1,0){4}}
\put(250,6.9){\line(0,1){5.5}}
\multiput(270,33)(10,0){2}{\line(0,1){6}}
\multiput(310,23)(10,0){2}{\line(0,1){6}}
\color{black}
\put(228,17.5){\scriptsize 1}
\put(228,27.5){\scriptsize 2}
\put(238,7.5){\scriptsize 3}
\put(238,17.5){\scriptsize 4}
\put(238,27.5){\scriptsize 5}
\put(248,7.5){\scriptsize 6}
\put(261,33.8){\scriptsize $3m+2$}
\put(301,23.8){\scriptsize $6m+7$}
\put(250,30){\circle*{2}}
\put(290,20){\circle*{2}}
\qbezier[12](250,30)(252,35)(259,36)
\qbezier[12](290,20)(292,25)(299,26)
%
\end{picture}
\caption{}\label{lat}
\end{figure}
%
Then we can find the gap sequence 
$$\{1,2,4,5,\ldots ,3m+1,3m+2,3m+4,\ldots ,6m+7\}$$
as a set of integers contained in Int\,$\square_{C}$. Since 
the genus of $C$ is equal to the number of lattice points contained 
in Int\,$\square_{C}$, we have $g=3m+4$ in this case. Hence 
the above gap sequence at a total ramification point $P$ is truly of 
type I in Theorem \ref{gap}. Similarly, for the remaining cases in 
Fig. \ref{fig37}, we obtain the gap sequence at a ramification point 
and the genus of $C$ as follows. 
\begin{eqnarray*}
\!\begin{array}{cl}
{\rm (ii)} &\!\!\!\! \{j\mid \mbox{the line $3X+2Y=6+j$ has a lattice 
point in Int\,$\square_{C}$}\},\ g=3m+3,\\[1mm]
{\rm (iii)} &\!\!\!\! \displaystyle \{j\mid \mbox{the line $2X+\beta 
Y=2\beta+j$ has a lattice point in Int\,$\square_{C}$}\},\ \displaystyle 
g=3m-\frac{\beta -9}{2},\\
{\rm (iv)} &\!\!\!\! \{j\mid \mbox{the line $2X+Y=3+j$ has a 
lattice point in Int\,$\square_{C}$}\},\ g=3m+3.
\end{array}
\end{eqnarray*}
Consequently, in each case, the result of Theorem \ref{gap} can be 
visualized in a similar way as in Fig. \ref{lat}. 

This idea is applicable for the cases of higher gonality. By 
Theorem \ref{mthm}, a lattice polygon associated to a $k$-gonal 
curve $C$ can be drawn 
as a polygon with height $k$ and sufficiently large width. Assume that 
there exists an oblique side which has no lattice points except for two 
end points, and denote by $D$ the $T$-invariant divisor corresponds to 
this side. In this case, a point $P=C\cap D$ is a total ramification 
point of a gonality pencil on $C$, and moreover, $P$ satisfies the 
assumption in Corollary 1.6 in \cite{kaw}. Hence one can determine the 
Weierstrass gap sequence at $P$ by moving the oblique side similarly to 
Fig. \ref{lat}. This fact suggests the possibility of the classification 
of gap sequences at total ramification points of a curve on a toric 
surface. We will deal with this prospective problem in future work. 
\\[7mm]
{\it Acknowledgements.}
The author would like to express his sincere gratitude to Professor 
Kazuhiro Konno for his helpful advice and suggestions. He also 
thanks Doctor Takeshi Harui for uncountable fruitful discussions. 
In revising the article, Professor Wouter Castryck made a significant 
observation. 
\address{
Center for Basic Eduction Support\\
Faculty of Engineering\\
Kyushu Sangyo University\\
Higashi-ku, Fukuoka 813-8503\\
Japan}
{kawa@ip.kyusan-u.ac.jp}
\end{document}